\documentclass[10pt,reqno]{amsart}
\usepackage[top=.9in, left=.9in, right=.9in, bottom=.8in]{geometry}

\usepackage{savesym}
\savesymbol{checkmark}
\usepackage{nccmath, soul, dingbat}
\allowdisplaybreaks
\numberwithin{equation}{section}
\usepackage{graphicx}
\graphicspath{ {./images/} }

\usepackage[normalem]{ulem}
\makeatletter 
\def\@tocline#1#2#3#4#5#6#7{\relax
  \ifnum #1>\c@tocdepth 
  \else
    \par \addpenalty\@secpenalty\addvspace{#2}%
    \begingroup \hyphenpenalty\@M
    \@ifempty{#4}{%
      \@tempdima\csname r@tocindent\number#1\endcsname\relax
    }{%
      \@tempdima#4\relax
    }%
    \parindent\z@ \leftskip#3\relax \advance\leftskip\@tempdima\relax
    \rightskip\@pnumwidth plus4em \parfillskip-\@pnumwidth
    #5\leavevmode\hskip-\@tempdima
      \ifcase #1
       \or\or \hskip 1em \or \hskip 2em \else \hskip 3em \fi%
      #6\nobreak\relax
    \dotfill\hbox to\@pnumwidth{\@tocpagenum{#7}}\par
    \nobreak
    \endgroup
  \fi}
\makeatother
\usepackage{transparent}
\usepackage{amsmath,amsthm,amssymb}
\usepackage[initials,msc-links]{amsrefs}
\usepackage{url}
\usepackage[dvipsnames]{xcolor}
\usepackage[english=american]{csquotes}
\usepackage{enumerate}
\usepackage{textcomp}
\usepackage{stmaryrd} 
\SetSymbolFont{stmry}{bold}{U}{stmry}{m}{n}
\usepackage{mathrsfs}
\usepackage{mathtools}
\usepackage{color, colortbl}
\definecolor{Gray}{gray}{0.9}
\usepackage{bbm}
\usepackage{hyperref}
\usepackage{stmaryrd}
\usepackage{xfrac}
\usepackage{wrapfig}
\hypersetup{
    unicode=false,          
    pdftoolbar=true,        
    pdfmenubar=true,        
    pdffitwindow=false,     
    pdfstartview={FitH},    
    pdftitle={My title},    
    pdfauthor={Author},     
    pdfsubject={Subject},   
    pdfcreator={Creator},   
    pdfproducer={Producer}, 
    pdfkeywords={keyword1, key2, key3}, 
    pdfnewwindow=true,      
    colorlinks=true,       
    linkcolor=blue ,          
    citecolor=blue ,        
    filecolor=magenta,      
    urlcolor=Aquamarine           
}
\usepackage{xcolor}
\usepackage{epsfig,color,graphicx,graphics}
\usepackage{caption}
\usepackage{amsfonts}
\usepackage{tikz}
\usepackage{siunitx}
\usepackage{booktabs,colortbl, array}
\usepackage{pgfplotstable}
\pgfplotsset{compat=1.8}

\definecolor{rulecolor}{RGB}{0,71,171}
\definecolor{tableheadcolor}{gray}{0.92}

\newtheorem{theorem}{Theorem}[section]
\newtheorem{lemma}[theorem]{Lemma}
\newtheorem{proposition}[theorem]{Proposition}
\newtheorem{corollary}[theorem]{Corollary}

\theoremstyle{definition}

\newtheorem{remark}[theorem]{Remark}
\newtheorem{example}[theorem]{Example}

\newcommand{\cc}{\subset\!\subset}

\newcommand{\Orm}{\mathrm{O}}

\newcommand{\Ecal}{\mathcal{E}}

\newcommand{\Gcal}{\mathcal{G}}
\newcommand{\Hcal}{\mathcal{H}}

\newcommand{\Pcal}{\mathcal{P}}

\newcommand{\Xcal}{\mathcal{X}}

\newcommand{\Sbb}{\mathbb{S}}

\DeclareMathOperator{\diam}{diam}

\DeclareMathOperator{\dist}{dist}

\newcommand{\N}{\mathbb{N}}
\newcommand{\R}{\mathbb{R}}

\newcommand{\loc}{\mathrm{loc}}

\newcommand{\spt}{\mathrm{spt}}

\newcommand{\eps}{\epsilon}
\newcommand{\cG}{\mathcal{G}}

\renewcommand{\eps}{\varepsilon}
\newcommand{\e}{\varepsilon}

\DeclareMathOperator{\Div}{div}
\DeclareMathOperator{\Lip}{Lip}
\newcommand{\mres}{\mathbin{\vrule height 1.6ex depth 0pt width
        0.13ex\vrule height 0.13ex depth 0pt width 1.3ex}}

\DeclareMathOperator{\Lawson}{Lawson}
\DeclareMathOperator{\lens}{lens}
\DeclareMathOperator{\plane}{plane}
\newcommand{\X}{\mathcal{X}}

\title[]{On the non-uniqueness of locally minimizing clusters via singular cones}

\author[L. Bronsard]{Lia Bronsard}
\address{Department of Mathematics, McMaster University, 1280 Main St West, Hamilton, ON L8S 4J8, Canada}
\email{bronsard@mcmaster.ca}

\author[R. Neumayer]{Robin Neumayer}
\address{Department of Mathematical Scienes, Carnegie Mellon University, 5000 Forbes Avenue, Pittsburgh, PA 15213, United States of America }
\email{neumayer@cmu.edu}

\author[M. Novack]{Michael Novack}
\address{Department of Mathematics, Louisiana State University, 303 Lockett Hall, Baton Rouge, LA 70803, United States of America}
\email{mnovack@lsu.edu}

\author[A. Skorobogatova]{Anna Skorobogatova}
\address{Institute for Theoretical Sciences, ETH Z\"{u}rich, Scheuchzerstrasse 70
8092 Z\"{urich}, Switzerland}
\email{anna.skorobogatova@eth-its.ethz.ch}

\begin{document}

\begin{abstract}
We construct partitions of $\mathbb{R}^n$ into three sets $\{\X(1),\X(2),\X(3)\}$ that locally minimize interfacial area among compactly supported volume preserving variations and that blow down at infinity to singular area-minimizing cones.
{As a consequence, we prove the} non-uniqueness of the standard lens cluster in a large number of dimensions starting from $8$. 
\end{abstract}

\maketitle

\section{Introduction}
\subsection{Overview}
The classical multiple bubble problem in $\R^n$ concerns the existence and structure of configurations of $N$ regions of  prescribed finite volumes, together with an exterior chamber of infinite volume, that minimize the surface area of their interfaces.\footnote{Note that the case $N=1$ is the classical isoperimetric problem.} This problem has received much attention in recent decades, with an abundance of work concerning the existence \cite{Alm76}, structure and regularity \cite{JTaylor76, ColEdeSpo22}, and characterization \cite{FoiAlfBroHodZim93, HutMorRitRos02,ReiHeiLaiSpi03,Rei08, Wichiramala,PaTo20, DRTi23,MilXu25, MilNee23} of minimizers and critical points. This includes recent breakthrough work in which  Milman \& Neeman \cites{MilNee22} gave a complete classification of minimizers in the case $N\leq \min\{4,n\}$, resolving a long-standing conjecture of \cite{SulMor96}. 
See the introduction of \cite{MilNee22} for a comprehensive review of the literature surrounding this problem.
\medskip

Questions of these types have recently been extended to configurations with multiple infinite chambers, following the introduction of the $\boldsymbol{(1,2)}${\bf -cluster problem}, and more generally, the $(N,M)$-cluster problem,  by the first author with Alama \& Vriend in \cite{AlaBroVri23} in connection to tri-block copolymers. 
A $(1,2)$-cluster is a partition of $\mathbb{R}^n$ into three chambers, each of which is a set of locally finite perimeter, generically denoted by $\mathcal{X}=(\mathcal{X}(1),\mathcal{X}(2),\mathcal{X}(3))$, where $0<|\mathcal{X}(1)|<\infty$ and $|\mathcal{X}(2)|=|\mathcal{X}(3)|=\infty$.
The cluster perimeter of a $(1,2)$-cluster $\X$ in a ball $B_\rho(x)$ is 
\begin{equation}
    \label{eqn: cluster perimeter}
    \Pcal(\X;B_\rho(x)) := \frac{1}{2}\sum_{i=1}^3 P(\X(i); B_\rho(x))\,.
\end{equation}
Here $P(\X(i); B_\rho(x))$ denotes the relative perimeter of $\X(i)$ in $B_\rho(x)$, which simply agrees with the surface area of $\partial \X(i)\cap B_\rho(x)$ when $\partial \X(i)$ is piecewise smooth and is defined in section~\ref{s:prelim} in general.
A $(1,2)$-cluster $\mathcal{X}$ is locally minimizing\footnote{When there are two or more infinite-volume chambers, we necessarily need to consider \emph{local} minimizers, due to the infinite perimeter contribution of the interfaces between neighboring pairs of infinite-volume chambers.} 
if, for every $\rho>0$,
\begin{align}\label{eq:definition of local min}
    \Pcal(\Xcal;B_\rho(0))\leq \Pcal (\Xcal'; B_\rho(0))
\end{align}
whenever $\Xcal(i) \Delta \Xcal'(i) \cc B_\rho(0)$ and $|\Xcal(i)| = |\Xcal'(i)|$ for $i=1,2,3$\footnote{As shown in \cite{BroNov24}, local minimality of a $(1,2)$-cluster $\X$ with respect to compactly supported variations satisfying $|\X(i)| = |\X'(i)|$ for $i=1,2,3$ is equivalent to local minimality taken against variations satisfying $|\X(i) \cap B_\rho(0)| = |\X'(i) \cap B_\rho(0)|$ for $i=1,2,3$, where the volume of the infinite chambers is required to be preserved locally instead of only globally.}. {Clusters arise in the modeling of various physical systems such as immiscible fluid configurations \cites{Whi86,Leo01} in bounded domains or, for our setting with more than one infinite chamber, in  tri-block copolymers \cites{AlaBroLuWan22, AlaBroLuWan25} in the small-volume “droplet” regime for multiple phases.}  {Recent progress has been made toward existence, classification,    and stability for $(1,2)$-clusters (and $(M,N)$-clusters more generally) in \cite{AlaBroVri23,NovPaoTor23, BroNov24,BonaciniCristoferiTopaloglu2025, MilXu25}.}

\medskip

In \cite{BroNov24}, the first and third authors give a complete classification of  locally minimizing   $(1,2)$-clusters for $n\leq 7$, proving that the \emph{standard lens cluster} $\X_{\lens}$ is the unique locally minimizing $(1,2)$-cluster in $\R^n$ modulo homotheties. They obtain the same result in higher dimensions ($n\ge 8$) under the assumption of planar growth at infinity. The standard lens cluster, which is locally minimizing in every dimension, is characterized by the properties that {$|\X_{\lens}(1)|=1$},  $\partial\X_{ {\lens}}(2)\cap \partial\X_{{\lens}}(3) \subset \{x_n=0\}$ and $\partial \X_{ {\lens}}(1)$ is the union of pair of equal-radii spherical caps meeting on $\{x_n=0\}$ with equal angles of $\frac{2\pi}{3}$ between the three interfaces. An analogous classification was previously shown for $n=2$  in \cite{AlaBroVri23}.\medskip

The aforementioned classification in $n\leq 7$ is fundamentally related to the fact that planes are the only area-minimizing hypercones in these dimensions \cite{Simons}. In view of the existence of non-planar area-minimizing hypercones when $n \geq 8$ \cite{BDGG}, the following question was raised in \cite{BroNov24} and \cite{NovPaoTor23}:
\[
\mbox{\it For $n\geq 8$, do there exist locally minimizing $(1,2)$-clusters besides the standard lens?}
\]
The main goal of this paper is to answer this open question in the affirmative for a large number of dimensions, starting from $n=8$. 
{Part of our proof relies on error estimates for rigorous numerical comparison of appropriate renormalized quantities and so our statement goes up to dimension $2700$ as will be explained below.
}

\begin{theorem}\label{thm: summary}
    Let $n \in \{ 8, \dots, 2700\}.$ There exists a locally minimizing $(1,2)$-cluster $\X$ that is not a standard lens.
\end{theorem}

\subsection{Main Results}
The basic scheme to construct the local minimizing $(1,2)$-clusters of Theorem~\ref{thm: summary} goes as follows. Fix $n\geq 8$ and let $K$ be a perimeter minimizing cone in $\R^n$ whose boundary is not a plane, e.g., the region $K = \{ x_1^2 +\dots +x_4^2 < x_5^2 + \cdots + x_8^2\}$  bounded by the Simons cone in $\R^8.$\medskip

For $R>0$ large, we set up an energy minimization problem in the class of $(1,2)$-clusters $\X$ with $\X(1) \subset B_{3R}$ of volume $1$ for which, outside of $B_{3R}$, the chambers $\X(2)$ and $ \X(3)$ coincide with $K$ and $K^c$ respectively. Morally one wants to minimize the cluster perimeter $\mathcal{P}(\X; B_{4R})$, though by instead minimizing the energy $\mathcal{P}(\X; B_{4R}) + \cG_R(\X(1))$ for a carefully constructed confinement potential $\cG_R$, we circumvent technical challenges that would arise from the possibility of $\X(1)$ saturating the constraint $\X(1) \subset B_{3R}$.
\medskip

For a sequence $R_k \to \infty$, take a sequence of minimizers $\X_k$ to this minimization problem with $R=R_k$. We wish to obtain a local minimizer {of $\Pcal$} as a limit of the $\X_k$, and to this end we use a ``partial concentration compactness approach,'' where the volume of $\X_k(1)$ is the quantity we wish to preserve in the limit. Recall how classical concentration compactness methods are often used to prove existence in variational problems with critical scaling or on non-compact domains: 
\medskip

\noindent{\it Step 1:} Characterize precisely how a sequence can lose compactness and the cost in energy to exhibit this behavior.\\
{\it Step 2:} Establish compactness for a (minimizing) sequence by showing its energy lies below this loss-of-compactness threshold.
\medskip

For instance, to produce constant scalar curvature metrics in a given conformal class in the Yamabe problem, one shows that a noncompact minimizing sequence for the Yamabe functional must concentrate at a point and pay precisely the Yamabe constant of the sphere in energy \cite{LeeParker}. Similarly, in the theory of stationary harmonic maps, {loss of strong $H^1$-compactness for equibounded sequences is characterized by the formation of ``bubbles", and the energy loss is realized exactly by the sum of energies of such bubbles; see \cite{LinRiviere02,N-V-energy-identity}.}

\medskip

In the present setting, Step 1 as stated seems out of reach, as there are in principle myriad possible asymptotic behaviors, and importantly, possible energy costs, of $\{\X_k\}$ if the sequence loses compactness {due to volume loss at infinity of $\X_k(1)$}. {The main reason for this is that in contrast to, e.g.,~ the classical cluster problem, there is an asymptotically infinite contribution to the energy coming from the minimal surface $\partial \X_k(2) \cap \partial \X_k(3)$, which may twist around slowly as it transitions from the boundary data $\partial K$ at scale $R_k$ and connects to the bounded components of $\X_k(1)$. In doing so, it may transition to resembling some other singular minimal cone with lower area density than $K$ at scales $\ll R_k$, and so there is no clear way to propagate information about this minimal surface down to the scale of $\X_k(1)$; see below Theorem \ref{t:higher-dim} for further discussion of this point. The possibility of this type of behavior is the main difficulty in the problem, and would need to be ruled out in order to characterize the minimizers obtained in Theorem \ref{thm: summary}.}
\medskip

One possible way for a piece of  $\X_k(1)$ to escape to infinity is along approximately planar portions of the minimal surface $\partial \X_k(2) \cap \partial \X_{k}(3)$. Using the rigidity result in \cite{BroNov24}, we deduce that in this case, $\X_k$ locally looks like a rescaling of the standard lens cluster. In particular, its local energy contribution in a large ball, after rescaling by the escaping volume of $\X_k(1)$, is approximately equal to sum of the area of an equatorial disk in this ball and the {\it renormalized lens energy}
\begin{equation}\label{eqn: Lambda plane def}
    \Lambda_{\plane}(n) := P(\X_{\lens}(1)) -\omega_{n-1}\rho_n^{n-1}\,. 
\end{equation}
Here $\rho_n$ denotes the radius of the disc $\X_{\lens}(1) \cap \{x_n=0\}$.

\medskip

{As described above, this} is not the only way that $\X_k(1)$ can lose mass at infinity. It may be the case that a piece of $\X_k(1)$ drifts off to infinity along a non-planar portion of the minimal surface $\partial \X_k(2) \cap \partial \X_k(3)$. 
The key observation is that the latter loss of compactness, after a translation, will 
 still produce a local minimizer as in Theorem~\ref{thm: summary} {\it with a singular blowdown cone}. Thus we do not need to show that the sequence $\{\X_k({1})\}$ is compact. We only need to rule out the possibility that all of the mass of $\X_k(1)$ is lost via asymptotic lens behavior---the one asymptotic behavior whose limiting energy we can characterize. 
\medskip

A basic competitor argument shows that $\mathcal{P}(\X_k;B_R)$ is bounded above by the sum of $P(K; B_R)$ and the cone constant
\begin{equation}
    \label{eqn: Lambda C def}
    \begin{split}
    \Lambda(\partial K) :=\inf \{P(\X(1)) - \mathcal{H}^{n-1}({\partial K} \cap \X(1)^{(1)}) :& \text{ clusters $\X$ with }|\X(1)|=1, \\ 
    &  \X(2) = K \setminus \X(1), \X(3) = K^c\setminus \X(1) \}\,.
    \end{split}
\end{equation}
In this notation, $\Lambda_{\rm plane}(n) = \Lambda(\mathbb{R}^{n-1}\times \{0\})$ and we use the notation $\Lambda_{\rm plane}(n)$ because it is more concise. Note that $\Lambda(\partial K) \leq \Lambda_{\rm plane}(n)$ by testing \eqref{eqn: Lambda C def} with a sequence of competitors exhibiting asymptotic lens behavior. When this inequality is strict, we {rule out the possibility of all mass escaping in the form of a lens and thus} prove that a piece of the limit of the $\X_k$ yields a locally minimizing $(1,2)$-cluster that is not a standard lens cluster:

\begin{theorem}\label{t:higher-dim}
    Let $n \geq 8$ and {$K$ be a singular perimeter minimizing cone}. If
    \begin{align}\label{eq:comp assump in thm}
    \Lambda({\partial K}) < \Lambda_{\plane}(n)    
    \end{align}
    then there is a locally minimizing $(1,2)$-cluster $\X$ of $\Pcal$ in $\R^n$ such that for every blowdown $K_\infty$ of $\X(2)$, $ \partial K_\infty$ is a singular area-minimizing cone.

    Moreover, for any such blowdown $K_\infty$, we have the density comparison
    \begin{equation}\label{e:cone-density-comparison}
        \Hcal^{n-1}(\partial K_\infty\cap B_1(0)) \leq \Hcal^{n-1}(\partial K\cap B_1(0))\,.
    \end{equation}
\end{theorem}
Since $\X(2)$ locally minimizes perimeter outside some compact set, all possible blowdowns have boundaries that are area-minimizing cones; we refer the reader to Section \ref{s:prelim} for details. However, Theorem \ref{t:higher-dim} does not guarantee that this cone only has an isolated singularity or other properties (e.g. as in \cite{Simon_cylindrical}) that would guarantee the uniqueness of blowdowns. Thus, blowdowns are implicitly taken along a subsequence of scales $\rho_j \uparrow +\infty$.
\medskip

The statement of Theorem \ref{t:higher-dim} leaves open a natural question: does the blowdown $K_\infty$ of the limiting minimizing cluster coincide with the cone $K$ that was used in the setup of the minimization problems on large balls $B_{R_k}$? Absent some symmetrization technique which allows one to characterize directly the minimizers for the problem on $B_{R_k}$ (which does not seem immediate; cf.~ Remark \ref{remark:symmetries}), the only way to conclude $K=K_\infty$ would be to ``propagate" information about $\partial \X_k(2) \cap \partial \X_k(3)$ from scale $R_k$ around the origin to an order $1$ scale around a portion of $\X_k(1)$, which may not be remaining close to the origin. Such an argument would require two non-trivial pieces of information in addition to \eqref{eq:comp assump in thm}: first, that the area ratios of $\X_k$ at scale $4R_k$ and at an order one scale comparable to $\diam \X_k(1)$ essentially coincide, and second, that proximity to $\partial K$ at scale $4R_k$ combined with near constancy of area ratios allows for ``propagation" of proximity to $\partial K$ down to an order one scale. The former piece of information only seems clear if one knew that $\partial K$ was the least density singular area minimizing hypercone in $\mathbb{R}^n$ (which is open however for all $n\geq 8$), and the latter is equivalent to establishing the uniqueness of tangent cones to minimal surfaces, which is also open outside of specific cases. {Characterizing the blowdown, even in specific cases where $K$ is, e.g.,~ a Lawson cone, as will be the case in Theorem \ref{t:Simons-vs-lens} below, seems to be the first step in understanding qualitative properties of the minimizers we construct in this paper. If this can be done, then natural questions arise regarding symmetries and other characteristics of minimizers; see Remark \ref{remark:symmetries} below for further elaboration on this point.}
\medskip

To make use of the conditional existence result of Theorem~\ref{t:higher-dim}, we must next verify that the strict inequality \eqref{eq:comp assump in thm} holds for certain choices of perimeter minimizing cones $K$.
We recall the family of \emph{Lawson cones}\footnote{Such cones are sometimes also referred to as \emph{quadratic cones} in the literature, since they are quadratic varieties.} {for $k,l\in\N$}
\[
    C_{k,l}:= \left\{(x,y) \in \R^{k+1}\times\R^{l+1} : |x|^2 = \frac{k}{l} |y|^2\right\}\,.
\]
The cones $C_{k,l}$ are area-minimizing cones with an isolated singularity in $\R^{k+l+2}$ when $k+l > 6$ and when $k+l=6$, the cones $C_{3,3}$ (Simons' cone) and $C_{2,4}$ are the only two area-minimizing Lawson cones in $\R^8$ (see \cite{Lawlor91}*{Section 5.1, Theorem 5.5.2}). We further define the quantity
\begin{equation}
    \label{eqn: Lambda Lawson def}
    \Lambda_{\Lawson}(n) := \inf\{\Lambda(C_{k,l}) : k,l\in \N, \ k+l+2=n\}\,.
\end{equation}

Our second main result is the validity of the hypothesis $\Lambda_{\Lawson}(n) < \Lambda_{\plane}(n)$ in {all sufficiently low dimensions $n \geq 8$}, thus allowing us to conclude that there exist locally minimizing $(1,2)$-clusters with singular blowdowns in those dimensions.

\begin{theorem}\label{t:Simons-vs-lens}
    For $n \in \{8,\dots, 2700\}$, we have $\Lambda_{\Lawson}(n) < \Lambda_{\plane}(n)$. 
    As a consequence, in these dimensions, there exist locally minimizing $(1,2)$-clusters that blow down to singular cones.
\end{theorem}
The first part of Theorem \ref{t:Simons-vs-lens} is established through an explicit construction of $(1,2)$-clusters to be used as competitors in the variational problems \eqref{eqn: Lambda Lawson def} and \eqref{eqn: Lambda C def}
to bound $\Lambda_{\Lawson}(n)$ from above. Then, the second part follows by an application of Theorem~\ref{t:higher-dim}.
The competitors are modeled on $C=C_{n/2-1,n/2-1}$ if $n$ is even, and modeled on $C=C_{(n-1)/2-1,(n-1)/2}$ if $n$ is odd. 
We then must compute $\Lambda_{\plane}(n)$ and  $P(\X(1))-\Hcal^{n-1}(C\cap\X(1))$ for these competitors $\X$, and compare the two quantities.
\medskip

In light of various symmetries, these computations may be reduced to one-dimensional integrals which, although can be expressed in terms of certain special functions, are not easy to estimate by hand to a desirable level of precision in all dimensions. Thus, part of the proof of Theorem \ref{t:Simons-vs-lens} is computer-assisted: we use the FLINT C library to provide us with a suitably precise calculation of these values in general. FLINT uses interval arithmetic and comes with rigorous error estimates; see \cite{GomezSerrano2019} for a survey on interval arithmetic and its role in computer-assisted proofs and numerical approximations with rigorous estimates.
\medskip

In even dimensions $n\in 2\N$, and in particular for $n=8$, we are able to obtain a reasonably clean closed formula for these integrals, thus allowing one in principle to compute both $P(\X(1))-\Hcal^{n-1}(C\cap\X(1))$ for our choice of competitor and $\Lambda_{\plane}(n)$ by hand {and avoid the need for computer assistance. We present these by-hand computations in detail for $n=8$ and obtain precise values for $\Lambda_{\rm plane}(8)$ and the estimate for $\Lambda_{\rm Lawson}(8)$, thereby providing a complete by-hand proof of Theorem \ref{t:Simons-vs-lens} for $n=8$; see Section~\ref{ssec: by-hand}}. 
\medskip

In recent years, there have been a number of instances of computer-assisted proofs for variational problems, such as \cites{DeSilvaJerison09,Mil24,DLGMV24}. The computer-assisted aspect of the work \cite{DeSilvaJerison09} bears particular resemblance to ours. Therein, the authors verify the existence of a singular minimizing conical solution to the one-phase Bernoulli free boundary problem on $\R^7$. This is established via a comparison to a suitable family of sub- and super-solutions to the one-phase Bernoulli problem. The construction of such families requires 
calculation of certain quantities which, as in the present setting, are difficult to perform by hand due to the presence of various special functions, and are therefore carried out in mathemtical computing software (in their case, Mathematica).

{
\begin{remark}
    During the preparation of this article, we have learned that Novaga, Paolini \& Tortorelli \cite{NovPaoTor25} independently have proven non-uniqueness of the standard lens cluster as a locally minimizing $(1,2)$-cluster.
\end{remark}}
\medskip
 
\subsection{Further Results and Discussion} We now discuss some further results shown in the paper and possible future directions.

\begin{remark}[A capillarity problem]
Using arguments similar to those of Theorem \ref{t:higher-dim}, we obtain minimizers for a capillarity problem on one side of singular minimizing cones with isolated singularities. This situation is simpler than Theorem \ref{t:higher-dim} because the cone is ``fixed", and so we defer precise statements and a description of the proof to Section \ref{sec:cap}.
\end{remark}

\begin{remark}[On possible symmetry properties]\label{remark:symmetries}
{One might expect that if a blowdown of a given locally minimizing $(1,2)$-cluster is a cone with certain symmetries, for instance $SO(k+1)\times SO(l+1)$ invariance in the case of $C_{k,l}$, then the cluster inherits the same symmetries, but this is a challenging and delicate problem. Aside from the obvious issue of a priori lack of uniqueness of blowdowns in general, note that even in the case where such uniqueness is known, such as when a blowdown has an isolated singularity, we generally expect much less symmetry than in the classical double bubble problem (and its weighted analogue). Indeed, in the latter case, one can obtain cylindrical symmetry due to the lack of boundary conditions (see \cite{Hut97}), whereas here our boundary data has merely $SO(k+1)\times SO(l+1)$ invariance. Similarly, in the work \cite{MilNee22}, the authors get spherical symmetry for the interfaces by exploiting convexity together with connectivity properties of neighboring components in the cluster (see Theorem 1.9 and Remark 1.11 therein). Moreover, in the work \cite{BroNov24} of the first and third authors, the symmetry properties of the minimizing cluster are obtained by symmetrizing with respect to a hyperplane and gluing this symmetrization into the cluster, from which the rigidity follows crucially from the precise expansion of area-minimizing hypersurfaces asymptotic to planes. The analogous expansion of area-minimizing surfaces asymptotic to singular cones do not appear to be sufficient to replicate such an argument, and symmetrization relative to Lawson cones is more complicated, in light of the presence of the chamber $\X(1)$, without which one could work with surfaces invariant under the symmetry group $SO(k+1)\times SO(l+1)$ as in \cite{Lawson72}.}
\end{remark}

\begin{remark}[Weighted local $(1,2)$-minimizers]
Weighted analogues of multiple bubble problems have also been studied, where the perimeter of the interface is replaced by a suitably weighted perimeter, see \cite{Law14}. In \cite{BroNov24}, the first and third authors showed that the ``weighted lens cluster" $\tilde{\X}$ (with interfaces given by two spherical caps and a portion of the plane $\{x_n=0\}$ up to rotation) locally minimizes the weighted energy
\begin{align*}
   \sum_{1\leq i <j\leq 3}c_{ij}\mathcal{H}^{n-1}(\partial^* \X(i) \cap \X(j))\qquad \qquad c_{ij}>0,\,\, c_{ij}<c_{ik}+c_{kj}\,,
\end{align*}
and that if additionally $c_{12}=c_{13}$, then uniqueness holds under a planar growth assumption. Therefore, the existence of local minimizers which blowdown to singular cones is natural in this setting also. The proof of Theorem \ref{t:higher-dim} should generalize to the case of local minimizers for the weighted energy as long as one replaces $\Lambda_{\rm plane}(n)$ with the suitable weighted quantity
\begin{align}\notag
  c_{12}\mathcal{H}^{n-1}(\partial^* \tilde{\X}(1)  \cap \partial^* \tilde{\X}(2)) + c_{13}\mathcal{H}^{n-1}(\partial^* \tilde{\X}(1)  \cap \partial^* \tilde{\X}(3)) - c_{23}\tilde{\rho}_n^{n-1}\,,
\end{align}
{where $\tilde \rho_n$ is the radius of the disc $\tilde \X(1) \cap \{x_0=0\}$} and similarly for $\Lambda_{\rm Lawson}(n)$. In addition, in the perturbative regime around the equal weights case, the weighted analogue of Theorem \ref{t:Simons-vs-lens} will clearly hold also since the relevant quantities are continuous in the weights. 
For the sake of concreteness, we do not address these questions here.
\end{remark}
\section{Notation and preliminaries}\label{s:prelim}
For a set of finite perimeter $E$ in $\R^n$, we let $P(E)$ denote its perimeter, $P(E; A)$ its relative perimeter in an open set $A\subset \R^n$, and $\partial^*E$ its reduced boundary. We refer the reader to \cite[Chaper 12]{Mag12} for the basics of sets of finite perimeter. In particular, in the definition of cluster perimeter \eqref{eqn: cluster perimeter}, in the summation we have 
\[
    P(\X(i); B_\rho(x)) = \Hcal^{n-1}(\partial^*\X(i)\cap B_\rho(x))
\]
 We will often refer to the boundaries $\partial^*\X(i) \cap\partial^*\X(j)$ between pairs of chambers as \emph{interfaces}.

 For any set $E$ of locally finite perimeter in the paper, we will always assume that we are working with a Lebesgue representative that satisfies 
\begin{align}\label{eq:norm choice}
\overline{\partial^* E} = \partial E\,,  
\end{align}
see \cite[Proposition 12.19]{Mag12}. This will ensure that density properties hold for the chambers of minimizing partitions.
{We let $A\Delta B$ denote the symmetric difference $(A\setminus B)\cup (B\setminus A)$ of two Borel sets $A$ and $B$.}

We will often be exploiting compactness of sequences of sets of finite perimeter with uniform perimeter bounds inside a compact set. We recall that this compactness is with respect to the metric $d(A,B) = \|\mathbf{1}_A - \mathbf{1}_B\|_{L^1}$, where $\mathbf{1}_A$ denotes the indicator on $A$ (see e.g. \cite{Mag12}*{Theorem 12.26}), and we will simply write $A_k \overset{L^1}{\to} A$ to denote the convergence of a sequence of finite perimeter sets $A_k$ with respect to this metric. If one merely has local uniform perimeter bounds, we instead obtain $L^1_\loc$ compactness, {and thus we write $A_k \overset{L^1_\loc}{\to} A$}. We analogously define the notion of $L^1$ and $L^1_\loc$ convergence for (1,2)-clusters, which we get under the assumption of a uniform or locally uniform $\Pcal$ upper bound respectively. We use the notation $\mathrm{o}_k(1)$ for terms that are converging to zero as $k \to +\infty$. We will write $B$ for the ball of volume $1$ centered at $0$ and $B_R$ for the ball of radius $R$ centered at $0$.

Given a local minimizer $\X$ of $\Pcal$ with $\Pcal(\X;B_\rho) \leq C \rho^{n-1}$, we recall that an $L^1_\loc$ limit along a subsequence of scales $\rho_j \uparrow +\infty$ for the rescaled sets
\[
    \frac{\X(2)}{\rho_j}
\]
is referred to as a blowdown of the set of {locally} finite perimeter $\X(2)$. Note that blowdowns not depend on the center, but may depend on the sequence $\{\rho_j\}$. If $\X$ is a locally minimizing $(1,2)$-cluster, then subsequential blowdowns always exist, and are locally minimizing $(0,2)$-clusters with conical interface. This follows from standard arguments; see e.g.~ \cite[Lemma 4.1, Corollary 4.6]{BroNov24}.

\section{Setup of minimization problem}\label{s:setup}
As described in the introduction, we will prove our main existence result Theorem \ref{t:higher-dim} via a compactness procedure for a sequence of minimizing $(1,2)$-clusters in balls of radii increasing to infinity. However, {we want to} avoid the challenges that would arise if  the bounded chamber $\X(1)$ for this sequence of minimizers attached to the boundary of the ball in which we are solving the minimization problem.  In order to circumvent this, we introduce a suitable penalization term to our energy along the sequence. We wish for the penalization functional to be chosen in such a way that it only penalizes $\X(1)$, and when comparing the energy of our sequence of minimizers to a sequence of competitors, the difference in the penalization terms decays to zero. This is needed to ensure that the limiting cluster is a local minimizer of $\Pcal$ without any penalization. 

With this in mind, we consider the functional
\[
    \cG_R(E;B_\rho(x)) := \int_{E \cap B_\rho(x)} g_R(|y|) \, dy\,,
\]
where
\begin{equation}\label{eq:gR def}
    g_R(t) := \begin{cases}
        \frac{t- {\sqrt{R}}}{\sqrt{R}} & t \geq {\sqrt{R}} \\
        0 & t \in [0, R)\,.
    \end{cases}
\end{equation}
We simply write $\cG_R(\X(1))$ to mean $\cG_R(\X(1); B_{4R}).$
The function $g_R$ is Lipschitz with Lipschitz seminorm $\tfrac{1}{\sqrt{R}}$. As a consequence, for any fixed $\rho>0$ and sequence of points $x_{k}$ with $|x_k| \leq 4R_k \uparrow +\infty$, the sequence $\{g_{R{k}}\}$
satisfies 
\begin{align}\label{eq:cons of lip vanishing}
   \sup_{y,z\in B_{4R_k}\cap B_\rho(x_k)}|g_{R_k}(y) - g_{R_k}(z) |\to 0\,. 
\end{align}
{This will be important for extracting a limiting local minimizer of $\Pcal$, as explained above.}
In turn let
\[
    \Ecal_R(\X; B_{4R}) := \Pcal(\X; B_{4R}) + {\cG_R(\X(1))}\,.
\]
We then consider the localized and penalized minimization problem
\begin{equation}\label{e:localized-min-problem}
     \mu_R := \inf \left\{\Ecal_R(\X; B_{4R}) : |\X(1)|=1, \ \X(2)\setminus B_{3R} = K\setminus B_{3R}, \ \X(3)\setminus B_{3R} = K^c\setminus B_{3R} \right\},
\end{equation}
for perimeter minimizing cone $K$ satisfying $\Lambda(\partial K) < \Lambda_{\plane}(n)$.
    
Note that the constraints in \eqref{e:localized-min-problem} in particular ensure that $\X(1) \subset B_{3R}$. Furthermore, observe that the existence of a minimizer for the problem \eqref{e:localized-min-problem} for each fixed $R$ follows by the the lower-semicontinuity of perimeter and the continuity of $\cG_R$ under $L^1$ convergence of clusters.

Let us collect some initial properties satisfied by minimizers $\X_R$ of $\mu_R$, beginning with various useful energy bounds. Taking as a competitor the $(1,2)$-cluster $\X'$ with $\X'(1)$ the volume-$1$ ball $B$ centered at the origin, $\X'(2) = K \setminus B$ and $\X'(3) = K^c \setminus B,$ we see that
\begin{equation}
    \label{eqn: energy bound 1}
    \Pcal(\X_R ;B_{4R}) \leq \Ecal_R(\X_R; B_{4R}) \leq \Ecal_R(\X'; B_{4R}) \leq P(B) + {P\big(K; B_{4R}\big)}.
\end{equation}
Together with the area minimizing property of the cone $\partial K$, this yields the (asymptotically sharp) bound
\begin{equation}
    \label{eqn: energy bound like Simons cone}
   P(K; B_{4R}) \leq  P(\X_R(h) ; B_{4R} ) \leq P(B) + {P\big(K; B_{4R}\big)} \quad \text{ for } h=2,3\,.
\end{equation}
{Noting that $\X_R(2) \setminus B_{3R}=K \setminus B_{3R}$ and thus $P(\X_R(2);B_{4R} \setminus B_{3R} )\geq P(K;B_{4R} \setminus B_{3R})$, this lower bound and the upper bound in \eqref{eqn: energy bound like Simons cone} yield
\begin{align}\label{eq:good bound on 3R}
    P(\X_R(2);B_{3R})\leq P(B) + P(K;B_{3R})\,.
\end{align}
}
Keeping in mind that $\X_R(1) \subset B_{3R}$, together with \eqref{eqn: energy bound 1} and the lower bound in \eqref{eqn: energy bound like Simons cone} yield
\begin{align}
    \notag
P(\X_R(1))  + {2} \cG_R(\X_R(1))&=2\mathcal{E}_R (\X_R;B_{4R}) - P(\X_R(2);B_{4R}) - P(\X_R(3);B_{4R})  \\ \notag
&\leq 2P(B) + 2P(K;B_{4R})- P(\X_R(2);B_{4R}) - P(\X_R(3);B_{4R}) \\ \label{eqn: energy bound sandwich}
&\leq 2 P(B)\,.
\end{align}
Using \eqref{eqn: energy bound sandwich}, we see how the penalization term $\cG_R$ forces confinement of $\X_R(1)$ to balls $B_{{\sqrt{R}}+s}$ in a measure sense: 
for any $s>0$, since $g(|y|) \geq s/\sqrt{R}$ for $y \in \X_R(1) \setminus B_{{\sqrt{R}}+s}$, from \eqref{eqn: energy bound sandwich} we have
\begin{equation}\label{eqn: measure confinement}
   |\X_R(1) \setminus B_{{\sqrt{R}}+s}|  \leq \frac{ \sqrt{R}\, P(B)}{s}\,.  
\end{equation}

{
Beyond \eqref{eqn: energy bound 1}, a different competitor gives us another energy upper bound, {which will come in useful later: 
\begin{equation}
    \label{eqn: energy bound from lens competitor}
 \mathcal{E}_R(\X_R;B_{4R}) \leq \Lambda_{\rm plane}(n) + \mathcal{P}(K,B_{4R}  )+ o_R(1)\,,
\end{equation}
where $o_R(1)$ converges to zero subsequentially as $R\to +\infty$. To see this, fix $z \in \partial^* K \cap B_{1/2}$, 
and let $z_R ={\sqrt{R}} \, z \in \partial^*K \cap B_{{\sqrt{R}}/2}$. Note that $(\partial K -z_R) \cap B_{20 \rho_n}$ converges smoothly to a hyperplane, subsequentially as $R\to +\infty$, which up to a rotation we may assume is $\{x_n=0\}$. We may therefore take a competitor that agrees with $K$ outside of the cylinder of radius $10\rho_n$ centered at $z_R$ and oriented by the plane $\{x_n=0\}$, with $z_R + \X_{\lens}(1)$ ``glued in" inside this cylinder. 
Noting that $\Gcal_R=0$ for this competitor by construction, the bound \eqref{eqn: energy bound from lens competitor} follows immediately.}
\\

Lastly, for any $B_\rho(x) \subset B_{4R}$, since ${\Lip(g_R)}\leq 1/\sqrt{R}$, 
another comparison argument but with a ball-shaped replacement inside $B_\rho(x)$ yields the local perimeter bound
\begin{align}\label{eq:first energy bound}
    \mathcal{P}(\X_R;B_\rho(x))
    \leq C(n)\rho^{n-1} + \frac{2\rho\,\max\{ 1, \omega_n \rho^n\}}{\sqrt{R}}\,.
\end{align}



\section{Compactness procedure}\label{s:compactness}
We will prove our main theorems via a compactness procedure for a sequence of minimizers $\X_k$ for the minimization problems \eqref{e:localized-min-problem} in $B_{4R_k}$ for scales $R_k \uparrow \infty$. Thus, throughout the remainder of this article, we will be considering a sequence $R_k \uparrow +\infty$, with associated minimization problems $\mu_{R_k}$ defined in \eqref{e:localized-min-problem}, and corresponding energies $\Ecal_{R_k}= \Pcal + \cG_{R_k}$. To simplify notation, we will henceforth let $\mu_k := \mu_{R_k}$, $\Ecal_k := \Ecal_{R_k}$ and $\cG_k := \cG_{R_k}$. The latter will have associated integrand $g_k := g_{R_k}$.

\subsection{Lower Density Estimates for $\X(1)$ and Consequences}
We prove lower volume density estimates for $\X_k(1)$ that are uniform in $k$. 
This lemma  has various important consequences, most importantly that the hard constraint $\X_k(1) \subset B_{3R_k}$ is not saturated.\\

\begin{lemma}\label{lem: density estimates}
    Let $\mathcal{X}_{k}$ be a minimizer of the variational problem $\mu_{k}$. There are constants $c_0>0$, $\rho_0\in (0,1]$ {and $N_0\in \mathbb{N}$}, independent of $k$, such that 
    \begin{align}
        \label{eqn: lower denisty estimate}
    |\mathcal{X}_k(1) \cap B_\rho(y)| &\geq c_0 \rho^n\qquad \mbox{for any } y\in \partial \X_k(1),\,\, \rho < \rho_0 \,.
        \end{align} 
{Moreover, there are points 
$\{x_{1,k},\dots ,x_{N_0, k}\} \subset B_{{ \Orm(\sqrt{R}_k)}}$
such that $\X_k(1) \subset \bigcup_{i=1}^{N_0} {B_4}(x_{k,i})$. In particular,}
\begin{align}\label{eq:nonsaturation}
    \X_k(1) \subset B_{{{\rm O}(\sqrt{R_k})}}\,.
\end{align}
\end{lemma}

Generally speaking, lower density estimates for minimizers of perimeter-driven functionals  (c.f.~ \cite[Theorem~16.14]{Mag12} or \cite{BroNov24}*{Lemma 4.1} in the setting of clusters) are proven using competitors $\X'$ with $\X'(1) = \X_k(1) \setminus B_r(y)$ for a point $y \in \partial \X_k(1)$ and scale $r< r_0$. For  volume-constrained problems such as \eqref{e:localized-min-problem}, to make $\X'$ admissible, the removed mass $\sigma:= |\X_k(1) \cap B_r(y)|$ must be added back to $\X'(1)$, and to achieve estimates with constants $C_0$ and $\rho_0$ independent of $k$,  the added energy must be controlled {\it uniformly in $k$}.

In order to obtain such a uniform volume fixing variation despite the fact that we have $\sup_k \|g_k\|_{C^0(B_{4R_k})} = + \infty$,  the first step of the proof of Lemma~\ref{lem: density estimates} is to show that a definite amount of mass of each $\X_k(1)$ concentrates in a region where $g_k$ is bounded uniformly in $k$, thereby allowing us to reintroduce the mass with a uniform error.\\

\begin{proof}[Proof of Lemmma \ref{lem: density estimates}]
The proof is divided into steps. In the first three steps we prove \eqref{eqn: lower denisty estimate}. 
Then in the fourth step we prove the ball containment and \eqref{eq:nonsaturation}.

\medskip

\noindent{\it Step 1: Nucleation.} 
Let $c(n)>0$ be the dimensional constant from the nucleation lemma \cite[Lemma 29.10]{Mag12}. Fix $0<\eps<\min \{1, \tfrac{P(\X_k(1))}{2nc(n)}\}$. Note that $\eps$ can indeed be chosen depending only on $n$, due to the isoperimetric inequality and the fact that $|\X_k(1)|=1$.
 Applying the nucleation lemma to $\X_k(1)$, and recalling \eqref{eqn: energy bound sandwich} and the fact that $|\X_k(1)|=1$, 
we may find points
$\{x_{1,k},\dots , x_{N_k,k}\}$, with $N_k\leq \big(\tfrac{P(\X_k(1))}{c(n)\eps}\big)^n \leq \big(\tfrac{2P(B)}{c(n)\eps}\big)^n$, such that
{
\begin{equation}
    \label{eqn: cover almost all}
  \Big|\X_k(1)\setminus \bigcup_{i=1}^{N_k} B_2(x_{k,i})\Big| < {\eps}. 
\end{equation}
and} for each $x_{k,i}$,
\begin{align}
    \label{eqn: nucleation 1}
  |B_1(x_{k,i}) \cap \X_k(1)| \geq \bigg(c(n) \frac{\e}{P(\X_k(1))} \bigg)^n \geq \bigg(c(n) \frac{\e}{2P(B)} \bigg)^n =:\alpha{(\e)}\,,
  \end{align}
where $B$ is the ball of volume $1$. On the other hand, letting $s = s_k = 2\sqrt{R_{k}}P(B)/\alpha{(\e)}$ in \eqref{eqn: measure confinement}, we have $|\X_k(1)\setminus B_{{\sqrt{R_k}}+s_k}| \leq \alpha{(\e)}/2$. As a consequence of this and \eqref{eqn: nucleation 1}, we must have
\[
    B_1(x_{k,i})\cap B_{{\sqrt{R_k}}+s_k} \neq \varnothing\,,
\]
and thus we conclude that
\[
B_{1}(x_{k,i}) \subset B_{{\sqrt{R_k}} + s_k+2} \qquad \text{ for each } \, i =1,\dots , N_k\,.
\]
Importantly, this guarantees that 
\begin{equation}
    \label{eqn: g bound}
g_k(|x_{k,i}|) \leq g_k({\sqrt{R_k}} + s_k + 1)=\frac{2P(B)}{\alpha}+\frac{1}{\sqrt{R_k}}\leq C
\end{equation}
for a constant independent of $k$.

In light of the uniform bound on $N_k$, up to passing to an unrelabeled subsequence, we may assume that each $N_k$ is equal to the same number $N_0$. Thanks to the uniform perimeter bound \eqref{eq:first energy bound}, for each $i=1,\dots, N_0$, the sequence $\X_k- x_{k,i}$ converges locally in $\mathbb{R}^n$ as $k\to \infty$ to a cluster $\mathcal{Y}_i$ (which may have an empty second or third chamber)
with locally finite perimeter. Furthermore, we have $|\mathcal{Y}_i(1)|>0$ by \eqref{eqn: nucleation 1}.

\medskip

\noindent{\it Step 2: Volume fixing variations.}
We will perform volume fixing variations in $\mathcal{Y}_1(1)$ using \eqref{eqn: g bound}.
By the relative isoperimetric inequality and $|\mathcal{Y}_1(1)|\leq 1$, $\mathcal{H}^{n-1}(\partial^* \mathcal{Y}_1(1) \cap \tilde{B}) >0$ where $\tilde{B}$ denotes the ball of volume $2$ centered at the origin.
Thus, since we also have the uniform perimeter bound \eqref{eq:first energy bound}, arguing as in \cite[Lemma 17.21]{Mag12} (see also \cite{NovTopIhs23}*{Lemma 2.3}), we get the following volume-fixing variation lemma: there exist $\sigma_0,r_0  \in (0,1]$ and $z_1,z_2\in B_2$ with $B_{2r_0}(z_1) \cap B_{2r_0}(z_2)=\varnothing$, all depending on $\mathcal{Y}_1(1)$ and $n$, such that for any $|\sigma|<\sigma_0$ $k\in \mathbb{N}$, and $m\in\{1,2\}$, we can find a $(1,2)$-cluster $\X_k^{\sigma,m}$ with $\X_k^{\sigma,m}(j)\Delta \X_k(j) \Subset B_{r_0}(x_{1,k}+z_m)$ for $j=1,2,3$ and
\begin{itemize}
   \item[(i)] $|\X_k^{\sigma,m}(1)| = |\X_k(1)| + \sigma,$
    \item[(ii)] $\Pcal(\X_k^{\sigma,m};B_{r_0}(x_{1,k}+z_m)) \leq \Pcal(\X_k;B_{r_0}(x_{1,k}+z_m)) + C_0(\mathcal{Y}_1,n) \sigma,$
    \item[(iii)]$\cG_k(\X_k^{\sigma,m}(1);B_{r_0}(x_{1,k}+z_m)) 
    \leq \cG_k(\X_k(1);B_{r_0}(x_{1,k}+z_m)) +C_1 \, \sigma\, $
\end{itemize}
where $C_1=C+2/3$ for the constant $C$ in \eqref{eqn: g bound}. In the last estimate we have used  the observation that 
\[
\cG_k(\X_k^{\sigma,m}(1);B_{r_0}(x_{1,k}+z_m)) \leq \cG_k(\X_k(1);B_{r_0}(x_{1,k}+z_m))+
    \int_{(\X_k^{\sigma,m}(1)\Delta \X_k(1))\cap B_2(x_{1,k})} g_k(|y|)\,dy\,,
\]
together with the fact that  $g_k(|y|) \leq   g_k(|x_{1,k}|) + \tfrac{2}{\sqrt{R_k}} \leq C_1$ on $B_2(x_{1,k})$ thanks to \eqref{eqn: g bound}.

\medskip

\noindent{\it Step 3: Competitor construction.}
We are now in a position to proceed as in \cite{BroNov24}*{Proof of Lemma 4.1}. We include the details here for the purpose of clarity, due to the presence of the unbounded penalization term $\cG_k$, which does not exist in \cite{BroNov24}. Let $\rho_0 \leq \omega_n^{-1}\sigma_0^{1/n}$ and fix $y\in \partial \X_k(1)$.  Choose any $\rho <\min\{\rho_0,r_0\}$ and any $k\in \mathbb{N}$. {By our choice of $\rho<r_0$, we must have $B_\rho(y) \cap B_{r_0}(x_{k,i}+z_m)=\varnothing$ for at least one of $m=1,2$.}
Without loss of generality assume {that this holds for $m=1$ and that}
\begin{equation}\label{e:larger-unbdd-chamber}
	\Hcal^{n-1}(\partial^*\X_k(2)\cap \partial^*\X_k(1)\cap B_{\rho}(y)) \leq \Hcal^{n-1}(\partial^*\X_k(3)\cap \partial^*\X_k(1)\cap B_{\rho}(y))\,.
\end{equation}
Let $\sigma {=m_k(\rho)} = |\X_k(1) \cap B_{\rho}(y)|$ and let $\X'_k$ be the cluster formed by taking $\X'_k(1) = \X_k^{\sigma,1}(1) \setminus B_\rho(y)$, $\X'_k(2) = \X_k^{\sigma,1}(2)$ and $\X_k'(3) = \X_k^{\sigma,1}(3) \cup (\X_k(1)\cap B_\rho(y))$. 
%
%

The volume fixing variation of Step 2 yields
\begin{align}\label{eq:pot estimate}
    \cG_k(\X_k'(1)) \leq C_1 m_k(\rho) + \cG_k(\X_k(1))\,.
\end{align}

Testing the minimality of $\X_k$ for $\mu_k$ against the competitor $\X_k'$ and using \eqref{eq:pot estimate} and the definition of $\X_k'$ on $B_{r_0}(x_{k,i}+z_1))$, we therefore arrive at
\begin{align*}
    \Pcal(\X_k;B_{4R_k}) &= \Ecal_k(\X_k;B_{4R_k}) - \cG_k(\X_k(1)) \\
    &\leq \Ecal_k(\X_k';B_{4R_k}) - \cG_k(\X_k(1)) \\
    &\leq \Pcal(\X_k';B_{4R_k}) + C_1 m_k(\rho) \\ 
    &\leq \Pcal(\X_k';B_{4R_k}\setminus B_{r_0}(x_{k,i}+z_1)) + \Pcal(\X_k;B_{r_0}(x_{k,i}+z_1)) + \Lambda m_k(\rho)\,,
\end{align*}
where $\Lambda = C_1 + C_0$, for $C_0$ as in property (ii) of the volume fixing variation. Exploiting the definition of $\X_k'$ and the facts that $\X_k(1)\subset B_{3R_k}$ and $\X_k$ and $\X_k'$ agree outside of $B_{\rho}(y)\cup B_{r_0}(x_{k,i}+z_1)$,
this in turn gives
\begin{align*}
	&\mathcal{H}^{n-1}(\partial^*\X_k(1)\cap \partial^*\X_k(2)\cap B_\rho(y) ) + \mathcal{H}^{n-1}(\partial^*\X_k(1)\cap \partial^*\X_k(3)\cap B_\rho(y) ) \\
	&\qquad\leq \mathcal{H}^{n-1}(\partial^*\X_k(1)\cap \partial^*\X_k(2)\cap B_\rho(y) ) + \Hcal^{n-1}(\X_k(1)^{(1)}\cap \partial B_\rho(y) )  + \Lambda m_k(\rho)\,.
\end{align*}
Recalling \eqref{e:larger-unbdd-chamber} and the definition of $m_k$, by applying the isoperimetric inequality,
this reduces to
\[
	\frac{1}{2}{n\omega_n^{{1}/{n}}}m_k(\rho)^{{n-1}/{n}} \leq \tfrac{1}{2}P(\X_k(1)  \cap B_\rho(y)) \leq \frac{3}{2} m_k'(\rho)  + \Lambda m_k(\rho)\qquad \mbox{for a.e. $\rho<\rho_0$}\,.
\]
Provided that we additionally take $\rho_0 $ small enough in terms of $n$ and $\Lambda$, we may additionally absorb $\Lambda m_k(\rho)$ on the left-hand side to obtain
\[
    {\tilde c(n)}  m_k(\rho)^{(n-1)/n} \leq m_k'(\rho)\,.
\]
Integrating this inequality and using {$m_k(\rho)>0$} for every $\rho>0$ (which follows from \eqref{eq:norm choice}), \eqref{eqn: lower denisty estimate} for $\X_k(1)$ follows immediately.

\medskip

\noindent{\it Step 4: Containment in balls}. 
Next, fix $0< \e' < \min\{ 1, P(B)/2nc(n), c_0 \rho_0^n/2, {\omega_n}/2\}$.
Apply the nucleation lemma \cite[Lemma 29.10]{Mag12} again with this (possibly smaller) choice of $\e'>0$
to find points $\{x'_{k,i}\}_{i=1}^{N_k'}$ with $N_k' \leq \big(\tfrac{2P(B)}{c(n)\eps'}\big)^n$ satisfying \eqref{eqn: cover almost all} and \eqref{eqn: nucleation 1}.
As in Step 1, we have $B_1(x_{k,i}')\subset B_{{\sqrt{R_k}}+s_k'+2}$ for each $k=1,\dots, N_k'$ where we now let $s_k'= 2\sqrt{R_{k}}P(B)/\alpha{(\e')}$. Moreover, again up to passing to an unrelabeled subsequence, we may assume $N_k'$ is constantly equal to some number $N'.$ 

We claim that
\begin{equation}\label{eqn: contained in balls of radius 3}
\X_k(1) \subset \bigcup_{i=1}^{N'} {B_4}(x_{k,i}').
\end{equation}
Indeed, if $x_0 \in \partial \X_k(1)$, then $x_0 \in \bigcup_{i=1}^{N_0} B_{3}(x_{k,i}'),$ since otherwise \eqref{eqn: lower denisty estimate} together with the fact that $\rho_0\leq 1$ implies
\[
 \Big|\X_k(1)\setminus \bigcup_{i=1}^{N'} B_2(x_{k,i}')\Big| \geq |\X_k(1) \cap B_{\rho_0}(x_0) | \geq c_0 \rho_0^n,
\]
contradicting \eqref{eqn: cover almost all} {for our choice of $\eps'$}. 
Thus, we have $\partial \X_k(1) \subset \bigcup_{i=1}^{N'} B_3(x_{k,i}')$.
{Now, suppose that there exists a point $y_0 \in \X_k(1) \setminus \bigcup_{i=1}^{N_0} B_{4}(x_{k,i}')$. Then, since  $\partial \X_k(1) \subset \bigcup_{i=1}^{N'} B_3(x_{k,i}')$, we must have $\dist(y_0, \partial \X_k(1)) \geq 1$. But this in particular implies that
\[
    B_1(y_0) \subset \X_k(1) \setminus \bigcup_{i=1}^{N'} B_3(x_{k,i}')\,,
\]
which once again contradicts \eqref{eqn: cover almost all}} { since $\e'< |B_1(y_0)|/2$.} 
This shows the containment claim of the lemma,  from which the diameter bound \eqref{eq:nonsaturation} follows directly.
\end{proof}

{ 
Having shown the diameter bound \eqref{eq:nonsaturation}, we can now establish upper and lower density estimates in $B_{2R_k}$ for each of the three chambers.
\begin{corollary}\label{cor: more density estimates}
    Let $\mathcal{X}_{k}$ be a minimizer of the variational problem $\mu_{k}$. 
    There are constants $c_1\in (0,1)$ and $\rho_1\in (0,1]$, independent of $k$, such that for any $h \in \{1,2,3\},$ $y \in \partial X_k(h) \cap B_{2R_k}$, and $\rho \in (0,\rho_1],$
        \begin{align}
c_1 \omega_n\rho^n    \leq  |\X_k(h) \cap B_\rho(y)| \leq (1-c_1)\omega_n \rho^n \label{e:lower-density-2-3}
        \end{align} 
\end{corollary}
\begin{proof}
    Thanks to conclusion \eqref{eq:nonsaturation} of Lemma~\ref{lem: density estimates} {(which in particular implies that $g_k$ is bounded uniformly in $k$ in a neighborhood of $\X_k(1)$)}, the argument of Steps 1-3 of Lemma~\ref{lem: density estimates} can be repeated with $\X(h)$, {$h\in \{2,3\}$,} in place of $\X(1)$ to derive the lower bound of \eqref{e:lower-density-2-3}. The reason we restrict to $y\in \partial \X_k(h) \cap B_{2R_k}$, $h\in \{2,3\}$, as opposed to \eqref{eqn: lower denisty estimate}, which is for any $y\in \partial \X_k(1)$, is that for $y$ nearby $\partial B_{3R_k}$, the boundary conditions only allow removal of $\X_k(h)\cap B_\rho(y)$ from {$\X_k(h)$} when $h=1$.  Next, notice that lower density estimates for all three chambers directly imply the corresponding upper density estimates: fix $x \in {\partial} \X_k(1).$ Then  $x \in {\partial} \X_k(h)$ for either $h=2$ or $h=3$; without loss of generality suppose it is $h=2.$ Then by the lower density estimate for $\X_k(2)$ and the fact that the chambers are disjoint, we have 
    \[
    |\X_k(1)\cap B_\rho(x)| \leq | B_\rho(x) \setminus \X_k(2)| = \omega_n\rho^n - |B_\rho(x) \cap \X_k(2)| \leq  (1-c_1)\omega_n\rho^n. 
    \]
    The analogous argument yields the upper density estimates for $h=2,3.$
\end{proof}
}

\begin{remark}[Euler-Lagrange equations/regularity]\label{rmk: regularity}
 Using the uniform volume fixing variations from the previous lemma and the Lipschitz bound for $g_k$, one can show that there are constants $\Lambda$ and $r_0$ such that for all $k$, the restriction of $\X_k$ to $B_{3R_k}$ satisfies the almost-minimizing property
    \begin{align}\label{eq:lambda min cluster}
     \Pcal (\X_k; B_{3R_k})\leq \Pcal(\X';B_{3R_k}) + \Lambda\sum_{i}|\X_k(i) \Delta \X'(i)|
    \end{align}
    whenever $\X'$ is a $(1,2)$-cluster with $\X_k(i) \Delta \X'(i)\cc B_r(y)$ for some $B_r(y)\subset B_{3R_k}$ and $r<r_0$ \cite[Equation 1.1]{MagLeo17}; alternatively, $S_k:=\cup_i \partial \X_k(i)$ is $(M,\e,\delta)$-minimizing in $B_{3R_k}$ in the sense of \cites{Alm76,JTaylor76}. As a consequence, \cite[Theorem 3.10]{ColEdeSpo22} applies and gives precise information about the top three strata of $S_k$ in terms of planes, $\mathbf{Y}$-cones, and $\mathbf{T}$-cones. Also, one may derive the Euler-Lagrange equations for $S_k$ incorporating the volume constraint and the potential $g_k$. However, we do not need these properties in this paper.
\end{remark}

\subsection{Concentrations of $\X_k$ and their limit minimization problems}\label{ss:conc}
Let $\X_k$ be minimizers for the problem $\mu_k$ as described in the preceding section. In order to complete our compactness procedure as $k\to \infty$, we first identify a uniform (in $k$) number of disjoint balls whose union contains $\X_k(1)$.

\begin{lemma}\label{lemma:identifying concentrations}
There exist $N\in \mathbb{N}$ and $R_0>0$ such that, up to a subsequence, the minimizers $\X_k$ of the variational problem $\mu_k$ satisfy the following property:

\smallskip

for each $k$, there are $x_{k,i}\in B_{{\Orm(\sqrt{R_k})}}$, $1\leq i \leq N$, such that the balls $B_{R_0}(x_{k,i})$ are pairwise disjoint and 
$\X_k(1) \subset \cup_{i=1}^N B_{R_0}(x_{k,i})$.
\end{lemma}

\begin{proof}
    Let $N_0$ and $\{x_{k,1}, \dots, x_{k,N_0}\}$ be as in Lemma~\ref{lem: density estimates}. 
 For $i, j \in \{1, \dots, N_0\}$, we say $i \sim j$ if $\limsup_k |x_{k,i} - x_{k,j}|<+\infty.$ 
 Observe that $\sim$ defines an equivalence relation on the indices $i$. {Up to a further unrelabeled subsequence, we may assume that for each $k$ there are $N\leq N_0$ many equivalence classes partitioning $\{1,\dots, N_0\}$. Furthermore, up to relabeling of the points $x_{k,i}$, we may assume that for each $k$, $x_{k,1},\dots,x_{k,N}$ are representatives for each of the $N$ distinct equivalence classes.} By definition of the equivalence relation, there must exist $R_0>0$ such that {for large enough $k$,} the balls $B_{R_0}(x_{k,i})$ for $1\leq i \leq N$ are pairwise disjoint and 
 $\cup_{j \in [i]} {B_4}(x_{k,j}) \subset B_{R_0}(x_{k,i})$ for each $1\leq i \leq N$. After restricting to a tail such that this property holds, we conclude that
 \[
 \X_k(1) \subset \bigcup_{i = 1}^N B_{R_0}(x_{k,i})
\]
as desired. 
\end{proof}

For each $i =1,\dots , N$ and $k \in \mathbb{N}$ define the translated cluster
\begin{equation}\label{eq:concentration definition}
    \X_{k,i} := \X_{k} - x_{k,{i}}\,.
\end{equation}
 We will refer to each cluster $\X_{k,i}$ as a {\it concentration}.

 \begin{lemma}\label{lem: limit local min}
     Let $\X_{k,i}$ be as in \eqref{eq:concentration definition}. Then, up to a subsequence, not relabeled, there exist partitions $\X_{\infty,i}= (\X_{\infty,i}(1), \X_{\infty, i}(2), \X_{\infty, i}(3))$ {of $\R^n$} by sets of locally finite perimeter with the following properties:

\smallskip

\noindent (i) setting $v_i:=|\X_{\infty,i}(1)|$,
\begin{align}\label{eq:vi sum to 1}
    \sum_{i=1}^N v_i = 1\,;
 \end{align}

\noindent (ii) for each $h=1,2,3$, $\X_{k,i}(h)$ and $\partial \X_{k,i}(h)$ converge locally in the Hausdorff sense to $\X_{\infty,i}(h)$ {and $\partial \X_{\infty,i}(h)$ respectively};

\noindent (iii) {
If $|\X_{\infty,i}(2)|= |\X_{\infty,i}(3)|=\infty,$} then $\X_{\infty,i}$ is a locally minimizing $(1,2)$-cluster in the sense \eqref{eq:definition of local min}.   
{If $|\X_{\infty,i}(2)|$ or $|\X_{\infty,i}(3)|$ is finite, then 
\begin{align}\label{eq:floater lower bound}
    \frac{1}{2}\sum_{j=1}^3 P(\X_{\infty,i}(j);B_{2R_0}) \geq n\omega_n^{1/n} v_i^{(n-1)/n}\,.
\end{align}
}
\end{lemma}

\begin{proof}The proof is divided into steps.

\medskip

\noindent{\it Step 1}: Here we obtain the limiting clusters $\X_{\infty,i}$ and prove (i) and (ii). By the local perimeter bound \eqref{eq:first energy bound} and the compactness for sets of locally finite perimeter, for each $1\leq i \leq N$ and $h=1,2,3$ we obtain sets of finite perimeter $\X_{\infty,i}(h)$ such that $\X_{k,i}(h)\overset{L^1_\loc}{\to}\X_{\infty,i}(h)$. Furthermore, by Lemma \ref{lemma:identifying concentrations} and the fact that $|\X_k(1)|=1$, it follows that if we set $v_i=|\X_{\infty,i}(1)|$, then
\begin{align}\notag
    \sum_{i=1}^N v_i = 1\,,
\end{align}
which is (i).

To prove the local Hausdorff convergence of the chambers, first note that the inclusion $\X_k(1) \subset B_{{{\rm O}(\sqrt{R_k})}}$ from \eqref{eq:nonsaturation} implies that $\dist(x_{k,i},\partial B_{2R_k})\to \infty$ for every $1\leq i \leq N$. 
Thus, the  density estimates of Corollary~\ref{cor: more density estimates} along the boundaries of the chambers $\X_k(h)$ on $B_{2R_k}$ imply that the chambers of each concentration $\X_{k,i}$ enjoy {upper and lower} density estimates along their boundaries on balls that exhaust $\mathbb{R}^n$ as $k\to \infty$. 
From here, standard arguments (e.g.~ \cite[Proof of Theorem 21.14, Step 4]{Mag12}) give the local Hausdorff convergence. Other than $\partial \X_{k,i}(1)$, however, we do not know that these boundaries intersect $B_{R_0}$ non-trivially.

\medskip

\noindent{\it Step 2.} It remains to show (iii). {First consider the case that $|\X_{\infty,i}(2)|=|\X_{\infty,i}(3)|=\infty$.} The argument {that $\mathcal{P}$ is locally minimized} is standard, so we only sketch the details. 
For $s>0$, consider a local variation $\X_{i}'$ of $\X_{\infty, i}$ with $\X_{i}' \Delta \X_{\infty, i} \cc B_s$ {and $|\X_{i}'(1)| = v_i$}. For $k$ sufficiently large, we know $B_s(x_{k, i}) \subset B_{{\rm O}(\sqrt{R_k})}$ (as in Step 1 of Lemma~\ref{lem: density estimates}). Using this and the local Hausdorff convergence of the interfaces of $ \X_{k, i} $ to $ \X_{\infty, i}$, one can modify $\X_{i}'$ into an admissible competitor $\X''_{k,i}$ for $\X_k$ with $\X_k \Delta \X_{k,i}'' \Subset B_s(x_{k, i})$, analogously to that in the proof of \cite{NovPaoTor23}*{Theorem 2.13} (see also \cite{Mag12}*{Theorem 21.14} and \cite{BroNov24}*{Corollary 4.6}). From here, taking $\X_{k,i}''$ as a competitor for the minimality of $\X_k$ in the problem \eqref{e:localized-min-problem} and passing to the limit establishes the local minimality of $\X_{\infty, i}$.
The key reason the limiting energy for which  $\X_{\infty, i}$ is locally minimizing is $\mathcal{P}$, with no residual contribution coming from the limit of the $\cG_k$, is  the fact that the Lipschitz constant $\Lip(g_{R_k}) \leq 1/\sqrt{R_k}$ tends to zero as $k \to \infty$. {It remains to prove the lower bound \eqref{eq:floater lower bound} when at least one of $\X_{\infty,i}(2)$ or $\X_{\infty,i}(3)$ has finite volume. Since $|\X_{\infty,i}{(1)}|=v_i$, the fact that $\X_{\infty,i}(1) \cc B_{2R_0}$ and isoperimetric inequality yield
\begin{align}\notag
    \frac{1}{2} \sum_{j=1}^3 P(\X_{\infty,i}(j);B_{2R_0}) \geq P(\X_{\infty,i}(1);B_{2R_0}) = P(\X_{\infty,i}(1))\geq n\omega_n^{1/n} (v_i)^{(n-1)/n}\,,
\end{align}
which is \eqref{eq:floater lower bound}.}
\end{proof}

\section{Blowdowns of limit concentrations}\label{sec: blowdown}
Thanks to Lemma \ref{lem: limit local min}(iii), for any sequence of scales $\rho_j \uparrow + \infty$, we may consider a subsequential limit (again locally in Hausdorff distance), not relabeled,
\begin{equation}\label{eq:Ki def}
    K_i := \lim_{j \to \infty} \frac{\X_{\infty,i}(2)\cap B_{\rho_j}}{\rho_j}\,.
\end{equation}
Note that $K_i$ implicitly depends on the sequence $\{\rho_j\}$. 
Observe that if  $|\X_{\infty,i}(2)|$ or $ |\X_{\infty,i}(3)|$ is finite,  then $K_i$ is either the empty set or $\R^n$, and $\partial K_i$ is empty.
\medskip

Let us first demonstrate the validity of the density bound \eqref{e:cone-density-comparison}. For this, we require the following almost-monotonicity for area ratios of the interfaces of our concentration regions.

\begin{lemma}\label{l:monotonicity-concentrations}
    For every $k$ and $i$, there are $\rho_{k,i}\in (R_0,2R_0)$, Radon measures $\sigma_{k}$ with $\sup_k \sigma_k(B_{3R_k})<\infty$ and $\spt\, \sigma_{k}\subset \cup_{j=1}^N\partial B_{\rho_{k,j}}(x_{k,j})$, and unit vector fields $\nu_{M_k}^{\rm co}$ such that the surfaces
$$
M_k := \partial^*\X_k(2) \cap \partial^* \X_k(3) \setminus \cup_{j=1}^N B_{\rho_{k,j}}(x_{k,j})
$$
satisfy
\begin{align*}
   \int_{M_k \cap B_{3R_k}}\Div_{M_k}X \,d\mathcal{H}^{n-1} = \int_{B_{3R_k}} X \cdot \nu_{M_k}^{\rm co}\,d\sigma_k \qquad \forall X \in C_c^\infty(B_{3R_k};\mathbb{R}^n) 
\end{align*}
and the following monotonicity-type property for each $i$:
\smallskip

{\it if $\rho_{k,i}<r_k<s_k<\dist(x_{k,i},\partial B_{3R_k})$, then}
\begin{align}\label{eq:large scale area bounds}
    \frac{\mathcal{H}^{n-1}(M_k \cap B_{r_k}(x_{k,i}))}{r_k^{n-1}} &\leq \frac{\mathcal{H}^{n-1}(M_k \cap B_{s_k}(x_{k,i}))}{s_k^{n-1}} \\ \notag
    &\qquad + C_4(n,N)\sigma_k(B_{3R_k})\bigg[\frac{1}{r_k^{n-2}} + \frac{1}{\min_{j\neq i}|x_{k,i}-x_{k,j}|^{n-2}}\bigg]\,.
\end{align}
\end{lemma}

\begin{proof}
To identify $\rho_{k,i}$, we first use the fact that $\rho\mapsto \mathcal{H}^{n-1}(M_k \cap B_\rho(x_{k,i}))=:g_{i}(\rho)$ is increasing (and thus differentiable for a.e.~ $\rho$) and bounded from above by $C_1(n)R_0^{n-1}$ for $\rho \in (R_0,2R_0)$ to identify radii $\rho_{k,i}\in (R_0,2R_0)$ such that $g_{i}'(\rho_{k,i})$ exists and $g_{i}'(\rho_{k,i}) \leq C_2(n,R_0)$. Then by a cut-off argument (see e.g. \cite[Eq.~ 4.17]{BroNov24}) based on the existence of $g_{i}'(\rho_{k,i})$ and the fact that $M_k$ has vanishing (distributional) mean curvature on $B_{3R_k}\setminus \cup_{i} \partial B_{\rho_{k,i}}(x_{k,i})$, it follows that the rectifiable $(n-1)$-varifolds $V_k:=\underline{v}(M_k,1)$ satisfy 
\begin{align}\label{eq:first var of Mk}
    \delta V_k (X) = \int_{M_k} \Div_{M_k}X\,d\mathcal{H}^{n-1} \leq C_3(n,N,R_0)\|X\|_{L^\infty}\qquad \forall X\in C_c^\infty(B_{3R_k};\mathbb{R}^n)\,.
\end{align}
Thus by Riesz's theorem, $\delta V_k$ is a vector-valued Radon measure which has polar decomposition $\nu_{V_k}^{\rm co}(\cdot) \sigma_k$, where $\sigma_k = |\delta V_k|$ and $\nu_{V_k}^{\rm co}$ is $\Sbb^{n-1}$-valued. Furthermore, by \eqref{eq:first var of Mk} and the fact that $M_k$ has vanishing mean curvature on $B_{3R_k}\setminus \cup_{i}\partial B_{\rho_{k,i}}(x_{k,i})$, we have
\begin{align}\label{eq:sigmak bounds}
    \sigma_k(B_{3R_k})\leq C_3(n,N,R_0)\qquad \mbox{and}\qquad \spt\, \sigma_k \subset \bigcup_{i=1}^N\partial B_{\rho_{k,i}}(x_{k,i})\,.
\end{align}
Then, fixing $i$, by applying Lemma \ref{lemma:monotonicity} at radii $\rho_{k,i}<r_k<s_k$ and using \eqref{eq:sigmak bounds}, we have
\begin{align}\notag
    \frac{ \mu_{V_k}(B_{r_k}(x_{k,i}))}{r_k^{n-1}} &\leq \frac{\mu_{V_k}(B_{s_k}(x_{k,i}))}{s_k^{n-1}} +  \int_{B_{s_k}(x_{k,i})}\bigg|\frac{1}{\max\{r_k,|y-x_{k,i}|\}^{n-1}}-\frac{1}{s_k^{n-1}}\bigg|\frac{|y-x_{k,i}|}{n-1}\,d\sigma_k \\ \notag
    &= \frac{\mu_{V_k}(B_{s_k}(x_{k,i}))}{s_k^{n-1}} + \int_{\partial B_{\rho_{k,i}}(x_{k,i})}\bigg|\frac{1}{r_k^{n-1}}-\frac{1}{s_k^{n-1}}\bigg|\frac{|y-x_{k,i}|}{n-1}\,d\sigma_k \\ \notag
    &\,+ \sum_{j\neq i}\int_{B_{s_k}(x_{k,j}) \cap \partial B_{\rho_{k,j}}(x_{k,j})}\bigg|\frac{1}{\max\{r_k,|y-x_{k,i}|\}^{n-1}}-\frac{1}{s_k^{n-1}}\bigg|\frac{|y-x_{k,i}|}{n-1}\,d\sigma_k \\ \notag
    &\leq  \frac{\mu_{V_k}(B_{s_k}(x_{k,i}))}{s_k^{n-1}} + C_4(n,N)\sigma_k(B_{3R_k})\bigg[\frac{1}{r_k^{n-2}} + \frac{1}{\min_{j\neq i}|x_{k,i}-x_{k,j}|^{n-2}}\bigg].
\end{align}

Since $V_k$ is multiplicity one, $\mu_{V_k} = \Hcal^{n-1}\mres M_k$, and thus we conclude.
\end{proof}

As an immediate consequence of Lemma \ref{l:monotonicity-concentrations}, we obtain the following.
\begin{corollary}\label{c:blowdown-density-bound}
    For any $i=1,\dots, N$, the cone $\partial K_i$ from \eqref{eq:Ki def} satisfies
    \[
        \Hcal^{n-1}(\partial^* K_i \cap B_1) \leq \Hcal^{n-1}(\partial^* K\cap B_1)\,.
    \]
   In particular, the conclusion \eqref{e:cone-density-comparison} of Theorem \ref{t:higher-dim} holds for $\X = \X_{\infty,i}$ with any choice of $i$.
\end{corollary}

\begin{proof}
By Lemma \ref{l:monotonicity-concentrations} and by construction of the points $x_{k,i}$, for any pair of scales $2R_0<r_k<s_k<\dist(x_{k,i},\partial B_{3R_k})$, the concentration $\X_{k,i}$ with associated
\[
    M_{k,i} := \partial^*\X_{k,i}(2) \cap \partial^* \X_{k,i}(3) \setminus \cup_{j=1}^N B_{\rho_{k,j}}(x_{k,j}-x_{k,i})
\]
satisfies
\begin{align*}
    \frac{\mathcal{H}^{n-1}({M_{k,i}} \cap B_{r_k})}{r_k^{n-1}} \leq \frac{\mathcal{H}^{n-1}({M_{k,i}} \cap B_{s_k})}{s_k^{n-1}} + C_5 \bigg[\frac{1}{r_k^{n-2}} + \frac{1}{\min_{j\neq i}|x_{k,i}-x_{k,j}|^{n-2}}\bigg]\,,
\end{align*}
where $C_5 = C_4\sup_k \sigma_{k}(B_{3R_k})$. Since $|x_{k,i}|\leq {\rm O}(\sqrt{R_k})$ according to Lemma \ref{lem: density estimates}, this estimate holds for $2R_0<r_k < s_k \leq 3R_k - {\rm O}(\sqrt{R_k})$.
By choosing $s_k = 3R_k - {\rm O}(\sqrt{R_k})$ and estimating the right hand side from above using energy bound \eqref{eq:good bound on 3R}, we obtain
\begin{align*}
    \frac{\mathcal{H}^{n-1}(M_{k,i} \cap B_{r_k})}{r_k^{n-1}} \leq \frac{P(K;B_{3R_k})+P(B)}{(3R_k-{\rm O}(\sqrt{R_k}))^{n-1}} + C_5 \bigg[\frac{1}{r_k^{n-2}} + \frac{1}{\min_{j\neq i}|x_{k,i}-x_{k,j}|^{n-2}}\bigg]\,,
\end{align*}
Fix any $r>2R_0$. {Letting $r_k=r$ for all $k$ and} taking $k \uparrow +\infty$, we therefore arrive at
\[
    \frac{\mathcal{H}^{n-1}(M_{\infty,i} \cap B_{r})}{r^{n-1}} \leq P( K ; B_1)  + \frac{C_5}{r^{n-2}}\,,
\]
where $M_{\infty, i} = \partial^* \X_{\infty,i}(2) \cap \partial^* \X_{\infty,i}(3) \setminus  B_{\rho_{\infty,i}}$, where $\rho_{\infty,i}$ is a subsequential limit of the sequence $\{\rho_{k,i}\} \subset (R_0,2R_0)$. Now take $r = \rho_j \uparrow +\infty$, namely, the sequence of scales generating the cone $\partial K_i$. Since $M_{\infty,i}$ agrees with $\partial^* \X_{\infty,i}(2) \cap \partial^* \X_{\infty,i}(3)$ outside of the fixed ball $B_{\rho_{\infty,i}}$, the convergence of $\X_{\infty,i}(2)$ to $K_i$ yields the desired conclusion.
\end{proof}
\section{Proof of Theorem \ref{t:higher-dim}}

We are now in a position to complete the proof of Theorem \ref{t:higher-dim}. In light of Corollary \ref{c:blowdown-density-bound} and Lemma \ref{lem: limit local min}, it remains to verify that for at least one of the limiting concentrations $\X_{\infty,i}$, every blowdown $K_i$ has an interface that is a singular cone.
In this proof, it will be useful to introduce terminology for the case when either $|\X_{\infty,i}(2)|$ or $|\X_{\infty,i}(3)|$ is finite; in this case we call $\X_{\infty, i}$ a \emph{floater}.

\medskip

\textit{Case 1:} There exists $i$ and a sequence of scales $\rho_j \uparrow +\infty$ such that $\partial K_i$ is singular, namely it is nonempty and not equal to an $(n-1)$-dimensional plane.

In this case, 
as observed at the beginning of Section~\ref{sec: blowdown}, we must have $|\X_{\infty,i}(2)|= |\X_{\infty,i}(3)| = \infty$. So, Lemma \ref{lem: limit local min}(iii) guarantees that the corresponding limiting concentration $\X_{\infty,i}$ is a locally $\Pcal$-minimizing $(1,2)$-cluster subject to the volume constraint $|\X_{\infty,i}(1)|=v_i\in (0,1]$. By the monotonicity formula, the density of $\X_{\infty,i}$ at infinity, given by
\[
    \Theta_\infty(\X_{\infty,i}) = \lim_{r\to\infty} \frac{\Hcal(\partial^*\X_{\infty,i}(2)\cap B_r)}{r^{n-1}}\,,
\]
exists, and since $\partial K_i$ is singular, $\Theta_\infty(\X_{\infty,i}) > (1 + \eta)\omega_{n-1}$, where $\eta$ is the threshold determined by Allard's Regularity Theorem \cite{Allard_72}*{Section 8}. In particular, since $\Theta_\infty(\X_{\infty,i})$ is independent of the sequence of scales going to infinity, it is not possible for $\X_{\infty,i}$ to have a blowdown along some sequence $r_j \uparrow +\infty$ whose interface is contained in a plane, otherwise we would have
\[
    \lim_{j\to\infty} \frac{\Hcal(\partial^*\X_{\infty,i}(2)\cap B_{r_j})}{r_j^{n-1}} = \omega_{n-1}
\]
along such a sequence. We therefore conclude that in this case, \emph{any} blowdown of $\X_{\infty,i}$ has an interface that is a singular cone. In this case the proof of Theorem~\ref{t:higher-dim} is complete by letting $\X= \X_{\infty, i}$.

\medskip

\textit{Case 2:} For each $i=1,\dots, N $, either $\X_{i,\infty}$ is a floater, or for  every sequence of scales $\rho_j \uparrow + \infty$, the interface $\partial K_i$ of the blowdown of $\X_{\infty,i}$ along $\{\rho_j\}$ is an $(n-1)$-dimensional plane. Note that the latter case means that $\X_{\infty, i}$ has planar growth at infinity and thus
 \cite{BroNov24}*{Theorem 2.9} implies that $\X_{\infty,i}$ is a (rescaled) standard lens cluster $\X_{\lens}^{v_i}$ centered at a point $y_i$ with $|\X_{\lens}^{v_i}(1)| = v_i$ and $|y_i|\leq R_0.$ Up to re-indexing, we may assume that for some $M\in \{0,\dots, N\}$,  $\X_{\infty,1}\,\dots \X_{\infty, M}$  are standard lens clusters and $\X_{\infty,M+1},\dots , \X_{\infty, N}$ are floaters, with the convention that $M=0$ means there are no standard lens clusters.

We will rule out Case 2 in two steps.

\noindent {\it Step 1:} We claim that $M=N=1$, namely, there is just one concentration, and that it is a standard lens cluster.
We will prove this using a concavity argument together with the minimality of $\partial K$ to show that if not, the energy $\mathcal{E}_k(\X_k;B_{R_k})$ must be too large, violating \eqref{eqn: energy bound from lens competitor}. 
First, {observe that 
\begin{equation}\label{eqn: Lambda floater}
    \Lambda_{\rm plane}(n) < n \omega_n^{1/n}.
\end{equation}
The inequality follows from the fact that the standard lens cluster is 
the only local minimizer with planar growth by \cite[Theorem 2.9]{BroNov24}.
}

Next, observe that by scaling we have that for a lens $\X_{\lens}^{v}$ centered at a point $y \in \R^n$,
\begin{equation}
    \label{eqn: scaling lens energy}
\Pcal(\X_{\lens}^{v};B_{3R_0}(y)) - \omega_{n-1} (3R_0)^{n-1} = \Lambda_{\plane}(n) v^{\frac{n-1}{n}}\,,
\end{equation}
for any $v>0$ small enough such that $\X_{\lens}^v(1) \subset B_{R_0}.$

Now, if $\X_{\infty, i}$ is a floater, set $y_i=0$, while if $\X_{\infty, i}$ is a volume-$v_i$ standard lens cluster, let $y_i$ be its center as above. Let
\[
    A_k := \bigcup_{i=1}^N B_{3R_0}(x_{k,i}+y_i)\,
\]
and observe that $B_{3R_0}(x_{k,i}+y_i) \supset B_{R_0}(x_{k,i}).$

In light of the $L^1_\loc$ convergence of the $\X_{k, i}$ to $\X_{\infty,i}$, as well Lemma~\ref{lem: limit local min}(iii), \eqref{eqn: Lambda floater}, \eqref{eqn: scaling lens energy}, and the strict concavity of $t \mapsto t^{\frac{n-1}{n}}$, we have 
\begin{equation*}
    \begin{split}
    \lim_{k\to\infty}& \Pcal(\X_k;A_k) - M \omega_{n-1} (3R_0)^{n-1}\\
    & \geq
    \sum_{i=1}^M \left(\Pcal(\X_{\lens}^{v_i};B_{3R_0}(y_i)) - \omega_{n-1} (3R_0)^{n-1}\right)
    +\sum_{i=M+1}^N \bigg(\frac{1}{2} \sum_{j=1}^3 P(\X_{\infty,i}(j); B_{3R_0})\bigg) \\
    & \geq  \Lambda_{\plane}(n) \sum_{i=1}^M v_i^{\frac{n-1}{n}} 
    + n\omega_n^{1/n} \sum_{i=M+1}^N v_i^{\frac{n-1}{n}}
     \\
    & \geq \Lambda_{\plane}(n) \sum_{i}v_i^{\frac{n-1}{n}} \geq 
 \Lambda_{\plane}(n) \Big(\sum_{i}v_i \Big)^{\frac{n-1}{n}}= \Lambda_{\plane}(n)\,,        
    \end{split}
\end{equation*}
On the right-hand side of the first line, $\X_{\lens}^{v_i}$ is the lens centered at $y_i.$ If $M<N$, i.e. if there are floaters, the first inequality in the final line is strict in light of \eqref{eqn: Lambda floater}, while if $N>1,$ the second inequality in the final line is strict.
 Thus, for the full energy of $\X_k$, we have
 \begin{align}
 \begin{split}
      \label{eqn: lb for part of energy}
     \mathcal{E}_k(\X_k;B_{4R_k})&\geq  \Pcal(\X_k;B_{4R_k}) 
     = \Pcal(\X_k;A_k) + \Pcal(\X_k; B_{4R_k}\setminus A_k)  \\
     & \geq  \Lambda_{\plane}(n) + \Pcal(\X_k; B_{4R_k}\setminus A_k) + M \omega_{n-1}(3R_0)^{n-1} +\mathrm{o}_k(1). 
     \end{split}
 \end{align}
Notice that 
\[
\Pcal(\X_k; B_{4R_k}\setminus A_k) = \mathcal{H}^{n-1}(\partial^*\X_k(2) \cap \partial^*\X_k(3) \cap ( B_{4R_k}\setminus A_k)).
\]
Next, for each $i=1,\dots , M$ we know that $\X_{\infty,i}$ is a standard lens cluster {of volume $v_i$} centered at $y_i\in \R^n$ with $|y_i|\leq R_0$.
So, for each such $i$, $\X_{k,i}$ converges locally in the Hausdorff distance to its corresponding lens, and in particular, in the annulus $B_{4R_0}(x_{k,i}+y_i) \setminus B_{3R_0}(x_{k,i}+y_i)$, the interface $\partial^*\X_k(2) \cap \partial^*\X_k(3)$ is arbitrarily close to a plane in the Hausdorff distance.
By extending $\X_k(2)$ and $\X_k(3)$ inside each $B_{3R_0}(x_{k,i}+y_i)$ appropriately, for instance by gluing in the relevant plane in $B_{cR_0}(x_{k,i}+y_i)$ and paying along the boundary of a cylinder of radius $cR_0$ centered at $x_{k,i}+y_i$ for a suitable choice of $c\in [1,3/2]$, we see that there are complementary sets $\X'_k(2)$ and $\X'_k(3)$ that agree with $\X_k(2)$ and $\X_k(3)$ respectively on $B_{4R_k} \setminus A_k$ 
and which have
\begin{equation}\label{eqn: fill in}
\mathcal{H}^{n-1}(\partial^*\X_k'(2) \cap \partial^*\X_k'(3) \cap B_{4R_k}) = \Pcal(\X_k; B_{4R_k}\setminus A_k) +M \omega_{n-1}(3R_0)^{n-1} + \mathrm{o}_k(1).
\end{equation}
Here we are also using the fact that for $k$ sufficiently large the $\partial^*\X_k(2)\cap \partial^*\X_k(3)$ interface does not intersect $B_{3R_0}(x_{k,i})$ for $i=M+1,\dots , N$, using the Hausdorff convergence of Lemma~\ref{lem: limit local min}(ii) and the definition of floater.
On the other hand, using the area-minimizing property of the cone $\partial K$, we have 
\begin{equation}
    \label{eqn: using cone min}
\mathcal{H}^{n-1}(\partial^*\X_k'(2) \cap \partial^*\X_k'(3) \cap B_{4R_k}) \geq P(K ; B_{4R_k}).
\end{equation}
Combining \eqref{eqn: fill in} and \eqref{eqn: using cone min}  with \eqref{eqn: lb for part of energy}, we find 
\[
 \mathcal{E}_k(\X_k;B_{4R_k}) \geq  \Lambda_{\plane}(n) + P(K; B_{4R_k}) +  \mathrm{o}_k(1),
\]
with strict inequality unless $N=M=1$. Combining this with the energy upper bound \eqref{eqn: energy bound from lens competitor} for $k$ sufficiently large, we conclude by contradiction that $M=N=1.$

\medskip

\label{subsec:comp for lensshaped concentration}
\noindent{\it Step 2:}
Now that we have a single concentration $\X_{\infty,1}$, we are in a position to exploit the fact that $\Lambda(\partial K) < \Lambda_{\plane}(n)$ to use the fact that $\X_{\infty,1}$ is a standard lens cluster to contradict the minimality of $\X_k$ with an argument similar to the one in Step 1 above. Recall that we write $y_1$ to denote the center of this standard lens cluster and that $B_{3R_0}(y_1) \supset B_{R_0}.$

Since $\Lambda(\partial K) < \Lambda_{\plane}(n)$, there exists a $(1,2)$-cluster $\bar\X$ with $|\bar\X(1)|=1$, $\bar\X(1)$ bounded, and $\bar\X(2)= K\setminus \bar{\X}(1)$, and $\bar\X(3)= K^c\setminus \bar{\X}(1)$ such that
\begin{equation}\label{eq:lawson beats plane cluster}
P(\bar\X(1)) - \Hcal^{n-1}({\partial K}\cap \bar\X(1)^{(1)}){+ \alpha} \leq  \Lambda_{\plane}(n)
\end{equation}
for some $\alpha>0$. We therefore have for large $k$
\begin{align*}
    \Pcal(\X_k;B_{4R_k}) &= \Pcal(\X_k; B_{3R_0}(x_{k,1} + y_1)) - \omega_{n-1} (3R_0)^{n-1} + \Pcal(\X_k; B_{4R_k}\setminus B_{3R_0}(x_{k,1}+ y_1)) + \omega_{n-1} (3R_0)^{n-1} \\
    &= \Lambda_{\plane}(n) + \mathrm{o}_k(1) + \Pcal(\X_k; B_{4R_k}\setminus B_{3R_0}(x_{k,1} +y_1)) + \omega_{n-1} (3R_0)^{n-1} \\
    &\geq \alpha+  \big(P(\bar\X(1)) - \Hcal^{n-1}({\partial K}\cap \bar\X(1)^{(1)})\big) + \Pcal(K; B_{4R_k}) + \mathrm{o}_k(1) \\
    &=\alpha +\Pcal(\bar\X;B_{4R_k}) +\mathrm{o}_k(1)\,.
\end{align*}
In the final inequality we have used \eqref{eq:lawson beats plane cluster} as well as a repetition of the same gluing argument giving us \eqref{eqn: fill in} above and the perimeter minimality of the cone $K$ as in \eqref{eqn: using cone min}.
However, since $\bar\X$ is an admissible competitor {for $\mu_k$ and $\Pcal(\bar \X;B_{4R_k})=\mathcal{E}_k(\bar\X_k;B_{4R_k})$} for large $k$, for $k$ sufficiently large we reach a contradiction to the minimality of $\X_k$ as desired.

\section{Proof of Theorem~\ref{t:Simons-vs-lens}}\label{s:Lambda-Lawson-vs-Lambda-plane}
In this section we prove Theorem~\ref{t:Simons-vs-lens}. {This relies on an explicit computation of the values of $\Lambda_{\plane}(n)$, together with an upper bound on $\Lambda_{\Lawson}(n)$ by $\Lambda(C_{n/2-1,n/2-1})$ when $n$ is even, and by $\Lambda(C_{(n-1)/2 -1,(n-1)/2})$ when $n$ is odd. The computation requires one to evaluate integral expressions, which may simplified using special functions, but are still difficult to evaluate explicitly by hand in general. We therefore use {the FLINT library to do our interval arithmetic. The reader who does not have this installed can use Mathematica to check the computations for small $n$. For the reader who wishes to try pen and paper calculations, we also describe how to do the computations entirely by hand in all even dimensions, and {carry them out explicitly for $n=8$}.}

\subsection{Computing $\Lambda_{\plane}(n)$}
Recall the quantity $ \Lambda_{\plane}(n)$
defined in \eqref{eqn: Lambda plane def}. Our first step is to derive the following exact expression for $\Lambda_{\rm plane}(n)$ in terms of special functions:
\begin{align}\label{eq:final lambda plane expression}
    \Lambda_{\plane}(n) 
    & = \omega_{n-1}^{\frac{1}{n}} \, \frac{ (n-1)
\Big(\frac{\sqrt{\pi}\, \Gamma\left(\frac{n-1}{2}\right)}{\Gamma\left(\frac{n}{2}\right)} - {}_2F_1\left(\frac{1}{2}, \, \frac{3-{n}}{2} ; \,  \frac{3}{2} ; \,  \frac{1}{4}\right)\Big) - \Big(\frac{\sqrt{3}}{2}\Big)^{n-1}}{\Big(\frac{\sqrt{\pi}\,  \Gamma \left(\frac{n+1}{2}\right)}{\Gamma\left(1+\frac{{n}}{2}\right)}-{}_2F_1\left(\frac{1}{2},\, \frac{1-{n}}{2};  \, \frac{3}{2};\, \frac{1}{4}\right) \Big)^{\frac{n-1}{n}}}\,.
\end{align}
Here ${ }_2 F_1(a,b;c;z)$ is the Gauss hypergeometric function. 

Let $A_1$ be a spherical cap of unit radius in the upper halfspace $\{x_n\geq 0 \}$ meeting the hyperplane $\{x_n=0\}$ with interior contact angle $\pi/3$ and let $A_2$ be its reflection across  $\{x_n=0\}$. Let $V_{\lens}(n)$ denote the volume of the region enclosed by $A_1 \cup A_2$. 
By the coarea formula, 
\begin{align}
\label{eqn: comp perim lens}
    \mathcal{H}^{n-1}(A_1) &= (n-1)\omega_{n-1} \int_{1/2}^1 \big(1-t^2\big)^{\frac{n-3}{2}}\,dt\\
  \label{eqn: comp v lens}  V_{\lens}(n) &= 2\omega_{n-1} \int_{1/2}^1 \big(1-t^2\big)^{\frac{n-1}{2}} \,dt 
\end{align}
Note that when $n\geq 3$ is odd, the integrands in both expressions above are polynomials and thus can easily be integrated by hand. However, since we will also consider even dimensions, we proceed in generality.
Since the intersection of $A_1$ and $A_2$ is an $(n-2)$-dimensional sphere in $\{x_n=0\}$ of radius $\frac{\sqrt{3}}{2}$, 
\begin{align}\label{eqn: C lens intermed}
\Lambda_{\plane}(n) &= \frac{2\mathcal{H}^{n-1}(A_1) - \omega_{n-1} \left(\frac{\sqrt{3}}{2}\right)^{n-1}}{V_{\lens}(n)^{\frac{n-1}{n}}}
\end{align}%
To write this expression more explicitly, we first compute the antiderivative of the integrand in the integral expression \eqref{eqn: comp v lens} for $V_{\lens}(n).$ 
The change of variables $s = t^2$ gives us 
\begin{align*}  
\int_{0}^x \left(1-t^2\right)^{\frac{n-1}{2}} \,dt & =\frac{1}{2} \int_0^{x^2} \left(1-s\right)^{\frac{n-1}{2}} s^{-\frac{1}{2}} \,ds = \frac{1}{2}\, {\rm B}\left(x^2;\, \mfrac{1}{2} , \,\mfrac{n+1}{2}\right)
\end{align*}
where ${\rm B}(z; a, b) = \int_0^z s^{a-1} (1-s)^{b-1}\,ds$ is the incomplete beta function; see e.g. \cite[\S 8.17(i)]{NIST:DLMF}. Therefore, keeping in mind that ${\rm B}(1;a,b)$ is the (complete) beta function satisfying the identify ${\rm B}(1; a,b) = \frac{\Gamma(a)\Gamma(b)}{\Gamma(a+b)}$ and that $\Gamma(1/2) = \sqrt{\pi}$, we evaluate the integral in the expression \eqref{eqn: comp v lens} for $V_{\lens}(n)$ as
\[
\int_{1/2}^1 (1-t^2)^{(n-1)/2} \,dt  
= \frac{1}{2}{\rm B}\left(x^2; \mfrac{1}{2} , \mfrac{n+1}{2}\right)\Big|_{1/2}^1
= \frac{1}{2}\frac{\sqrt{\pi}\,\Gamma(\frac{n+1}{2})}{\Gamma(\frac{n+2}{2})} - \frac{1}{2}{\rm B}\left(\mfrac{1}{4}; \mfrac{1}{2} , \mfrac{n+1}{2}\right)\,.
\]
To further simplify the final term, we apply the identity  
\[
{\rm B}(z; a, b)=\frac{z^a}{a}\, { }_2 F_1\left(a, 1-b ; a+1 ; z\right) ;
\]
see \cite[\S 8.17(ii)]{NIST:DLMF}.  This leads us to the final expression
\[
V_{\lens}(n) =
\omega_{n-1}\bigg(\frac{\sqrt{\pi}\, \Gamma \left(\frac{n+1}{2}\right)}{\Gamma\left(1+\frac{{n}}{2}\right)}-{}_2F_1\left(\mfrac{1}{2}, \, \mfrac{1-n}{2}; \, \mfrac{3}{2}; \,  \mfrac{1}{4}\right)\bigg)
\]

Analogously evaluating the integral in the expression \eqref{eqn: comp perim lens} for $\mathcal{H}^{n-1}(A_1)$ shows 
\[
\mathcal{H}^{n-1}(A_1) = 
\frac{(n-1)\omega_{n-1} }{2}\bigg(\frac{\sqrt{\pi}\, \Gamma\left(\frac{n-1}{2}\right)}{\Gamma\left(\frac{n}{2}\right)} - {}_2F_1\left(\mfrac{1}{2},\, \mfrac{3-{n}}{2} ; \, \mfrac{3}{2} ;  \,\mfrac{1}{4}\right)\bigg)
\]
where again ${}_2F_1$ is the Gauss hypergeometric function. Substituting these expressions for $V_{\lens}(n)$ and $\mathcal{H}^{n-1}(A_1)$ into \eqref{eqn: C lens intermed} yields \eqref{eq:final lambda plane expression}.

\subsection{Estimating $\Lambda_{\Lawson}(n)$}
\label{ss:Simons}
Next, we estimate $\Lambda_{\Lawson}(n)$ from above via an explicit competitor construction. 
Choose two positive integers $k,l$ such that $n= k+l+2.$ 
Below we construct a set of finite perimeter $E$  in $\R^n$ with finite volume depending on $k$ and $l$.
Letting 
\[
    K = \{x_1^2 + \cdots + x_{k+1}^2 < \tfrac{k}{l}(x_{k+2}^2 + \cdots + x_n^2)\}
\]
denote the region bounded by the Lawson cone $C_{k,l}$, we then let $\X$ be the $(1,2)$-cluster with $\X(1) = E /|E|^{1/n}$ (so that it has the same volume as $\X_{\rm lens}(1)$), $\X(2) = K \setminus \X(1)$ and $\X(3) = (\R^n\setminus K)\setminus \X(1)$. 
For this cluster, we set
\begin{equation}
    \label{eqn: def M k l}
M(k,l) = P(\X(1)) - \mathcal{H}^{n-1}(C\cap \X(1))
\end{equation}
Then, keeping in mind the definition of  $\Lambda_{\Lawson}(n)$ in \eqref{eqn: Lambda Lawson def} and  by taking $\X$ as a competitor in the variational problem \eqref{eqn: Lambda C def} for $C= C_{k,l}$, we have 
\begin{equation}
    \label{eqn: Lambda Lawson leq M}
\Lambda_{\Lawson}(n) \leq M(k,l)\,.
\end{equation}

We define the set $E$ be invariant under the action of $SO(k+1)\times SO(l+1)$ on $\R^n$. Thus,  it suffices to define its ``slice" $\hat{E}$ in a two-dimensional quadrant $Q =\{ (u,v) : u>0, v>0\}$, where $u=|(x_1,\dots,x_{k+1})|$ and $v= |(x_{k+2},\dots,x_n)|$.  Letting $\lambda:= \sqrt{{k}/{l}}$, notice that 
\[
    \partial K \cap Q = \{(u,v): u>0\,, \ v>0\,, \ u = \lambda v\}\,.
\]

 In the quadrant, define $\hat E \cap \{u< \lambda v\}$
 to be the circular arc of radius
 $r>0$ and center $(0, -h)$ that forms angle $\pi/3$ with the line $\{u=\lambda v\}$ and angle $\pi/2$ with the $v$-axis $\{u=0\}$.
 We may parameterize this arc as the graph {over the $v$-axis} of the function $f: [1,-h+r] \to [0,\lambda]$ given by $f(v) = \sqrt{r^2-(v+h)^2}$. Let $\theta = \arctan(\lambda)$ denote the angle that the line $\{u=\lambda v\}$ makes with the $v$-axis $\{u=0\}$. 
Since $f(1) = \lambda = \sqrt{r^2 -(1+h)^2}$ and $f'(1)= -\tan(2\pi/3-\theta) = -\frac{1+h}{\lambda}$, this yields
\begin{equation}
    \label{eqn: r and h}
    h= \lambda \tan(2\pi/3-\theta) -1, \qquad 
    r = \lambda \sec(2\pi/3-\theta)\,.
\end{equation}

We define $\hat E \cap \{u> \lambda v\}$ analogously. Namely, we parameterize the arc of a circle of radius $\rho>0$ and center {$(-d,0)$} that forms angle $\pi/3$ with the line $\{u=\lambda v\}$ and angle $\pi/2$ with the $u$-axis $\{v=0\}$, via a function $g:[\lambda,\rho-d] \to [0,1]$ {defined by $g(u) = \sqrt{\rho^2 - (u+d)^2}$}. This time, we have $g(\lambda) = 1 = \sqrt{\rho^2-(\lambda+d)^2}$ and $g'(\lambda) = -\tan(\pi/6 +\theta)= -(\lambda+d)$, which in turn yields
\begin{equation}
    \label{eqn: rho and d}
    d= \tan(\pi/6 +\theta) - \lambda\,, \qquad 
    \rho = \sec(\pi/6 +\theta)\,.
\end{equation} 
See Figure \ref{uv-plot} below for a diagram of the part of $\hat E$, together with the cone $C_{k,l}$, lying in the upper-right quadrant of the $(u,v)$-plane.

\vspace{10pt}

\begin{figure}[hbtp]
\centering\includegraphics[height=200pt]{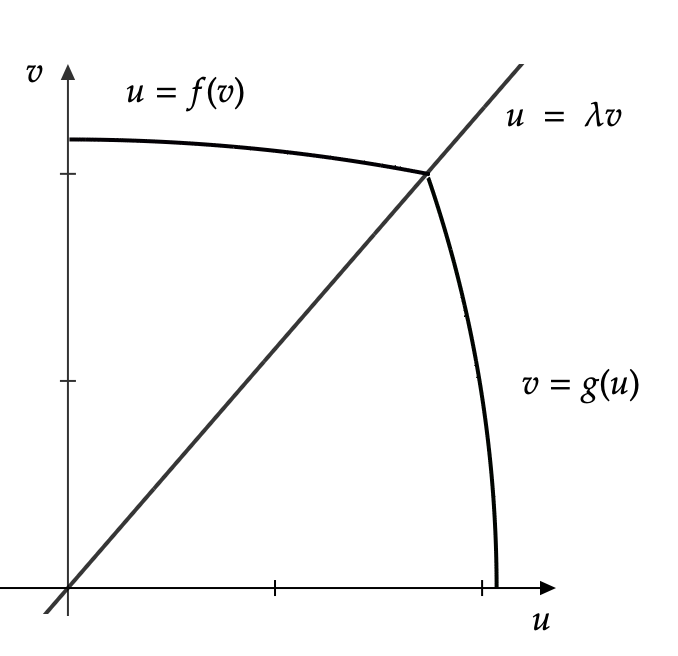}
\caption{\small A depiction of the competitor $\hat E$ in the upper-right quadrant of the $(u,v)$-plane.}
\label{uv-plot}
\end{figure}

The coarea formula (see e.g. \cite{Davini04}) yields
\begin{align}
\label{eqn: lawson comp perim 1}    \Hcal^{n-1}(\partial E) &= (k+1)(l+1) \omega_{k+1}\omega_{l+1}\int_{\partial \hat E} u^k v^l \, d\Hcal^1 \\
\label{eqn: lawson comp vol 1}    |E| &= (k+1)(l+1) \omega_{k+1}\omega_{l+1} \int_{\hat E} u^k v^l \, du\, dv\,.
\end{align}
Since, by the same formula,
\[
    P(K; E) = \lambda^k \sqrt{1+\lambda^2} (k+1)(l+1) \omega_{k+1}\omega_{l+1} \int_0^1 v^{k+l}\, dv = \frac{(k+1)(l+1) \omega_{k+1}\omega_{l+1}\lambda^k \sqrt{1+\lambda^2} }{k+l+1}\,,
\]
the quantity $M(k,l)$ defined in \eqref{eqn: def M k l} is given by
\begin{equation}\label{eqn: Lambda Lawson bounds}
    M(k,l) = \frac{1}{|E|^{\frac{n-1}{n}}} \bigg(\Hcal^{n-1}(\partial E)-\frac{(k+1)(l+1) \omega_{k+1}\omega_{l+1}\lambda^k \sqrt{1+\lambda^2} }{k+l+1} \bigg)\,.
\end{equation}

The right-hand side of \eqref{eqn: Lambda Lawson bounds} is made more explicit   by substituting $|E|$ and  $\Hcal^{n-1}(\partial E)$ for the following expressions, which follow from plugging definition of $\hat{E}$ into \eqref{eqn: lawson comp perim 1} and \eqref{eqn: lawson comp vol 1}:
\begin{align}
 \nonumber   |E|  = \omega_{k+1} \omega_{l+1} \Big( \lambda^{k+1} + (k+1) &\int_\lambda^{\rho -d} u^k \big(\rho^2 - (u+d)^2 \big)^{\frac{l+1}{2}} \,du\\
    &
\label{eqn: Lawson comp Flint vol}
+ (l+1)  \int_1^{r-h} v^l \big(r^2 - (v+h)^2\big)^{\frac{k+1}{2}} \,dv   \Big),\\
\nonumber    \mathcal{H}^{n-1}(\partial E)  = \omega_{k+1} \omega_{l+1} (k+1)(l+1)\Big( \rho & \int_{\lambda}^{\rho-d} u^k \big(\rho^2 -(u+d)^2 \big)^{\frac{l-1}{2}} \,du\\
\label{eqn: Lawson comp Flint perim}   
& + r  \int_{1}^{r-h} v^l \big(r^2 -(v+h)^2\big)^{\frac{k-1}{2}} \,dv  \Big)\,.
\end{align}

In fact, {by analogous reasoning to that for $\Lambda_{\plane}$ in the preceding section, recalling in addition the recursion formula $(k+1)\Gamma(k+1) = \Gamma(k+2)$} both of these expressions (and thus $M(k,l)$) can be expressed explicitly in terms special functions as follows:
\begin{align}\label{eqn: lawson competitor volume}
|E| &=  \omega_{k+1}\omega_{l+1}\left( {\lambda^{k+1}} + \big(\rho^2 - d^2\big)^{\frac{l+1}{2}} \, A +{(r^2-h^2)}^{\frac{k+1}{2}}B \right),\\
\label{eqn: lawson competitor perimeter}
\mathcal{H}^{n-1}(\partial E)  & = \omega_{k+1}\omega_{l+1} \left( (l+1)\rho \big( \rho^2 - d^2 \big)^{\frac{l-1}{2}} C + (k+1)r \big(r^2-h^2 \big)^{\frac{k-1}{2}}  D\right)
\end{align}
where 
\begin{align}
\begin{split}
\label{eqn: term A} A& =  
    \frac{ (\rho -d)^{k+1} \, \Gamma(k+2)\, \Gamma(\frac{l+3}{2})}{\Gamma(k+1 + \frac{l+3}{2})} \,  {}_2F_1\Big(-\mfrac{l+1}{2},\, k+1, \,\mfrac{l+1}{2} + k +2 ; \, -\left( \mfrac{\rho - d}{\rho +d}\right) \Big)\\
     &\qquad \qquad \qquad \qquad \qquad -{\lambda^{k+1}}\,F_1 \Big(k+1 ,\,  -\mfrac{l+1}{2},\,  -\mfrac{l+1}{2},\,  k+2; \, \mfrac{\lambda}{\rho- d} ,\,  -\mfrac{\lambda}{\rho +d}\Big) ,
   \end{split} 
\end{align}
\begin{align}
\begin{split}
  \label{eqn: term B}   B& = 
    \frac{ (r-h)^{l+1} \, \Gamma(l+2)\, \Gamma (\frac{k+3}{2})}{\Gamma(l+1 +  \frac{k+3}{2})} \, {}_2F_1 \Big( -\mfrac{k+1}{2},\, l+1,\, \mfrac{k+1}{2} + l +2 ; \, - \left(\mfrac{r-h}{r+h}\right)\Big)\\
  & \qquad \qquad \qquad \qquad \qquad \qquad
    - F_1\Big(l+1,\,  -\mfrac{k+1}{2}, \, -\mfrac{k+1}{2},\,  l+2; \, \mfrac{1}{r-h}, \, -\mfrac{1}{r+h}\Big)\,,
\end{split}
\end{align}
\begin{align}
\begin{split}
    \label{eqn: term C} 
        C & = \frac{(\rho-d)^{k+1} \Gamma(k+2)\Gamma(\frac{l+1}{2})}{\Gamma(k+1+ \frac{l+1}{2})} \, {}_2F_1\Big(\mfrac{1-l}{2} , k+1; \mfrac{l-1}{2} + k+2;\, -\left( \mfrac{\rho -d}{\rho+ d}\right) \Big)\\
    & \qquad \qquad\qquad \qquad \qquad - \lambda^{k+1}F_1\Big(k+1, \, \mfrac{1-l}{2}, \, \mfrac{1-l}{2}, \, k+2 ; \,\mfrac{\lambda}{\rho- d}, \, - \mfrac{\lambda}{\rho+ d} \Big)\,,
\end{split}
\end{align}
\begin{align}
\begin{split}
\label{eqn: term D} 
    D & =  \frac{(r-h)^{l+1} \Gamma(l+2)\, \Gamma(\frac{k+1}{2})}{\Gamma(l+1+\frac{k+1}{2})} \, {}_2F_1\Big(\mfrac{1-k}{2}, l+1; \mfrac{k-1}{2} + l+2; -\left(\mfrac{r-h}{r+h}\right) \Big)\\
    &
 \qquad \qquad \qquad\qquad\qquad \qquad- F_1\Big(l+1 , \mfrac{1-k}{2},  \mfrac{1-k}{2}; l+2; \mfrac{1}{r-h} ; -\mfrac{1}{r+h}\Big).
\end{split}
\end{align}
Here, ${}_2F_1$ once again denotes the Gauss hypergeometric function, and $F_1$ is the first Appell function, {and we are exploiting their respective integral representations}.
The proofs of the identities \eqref{eqn: lawson competitor volume} and \eqref{eqn: lawson competitor perimeter} are in the same spirit as the proof of the expression \eqref{eq:final lambda plane expression} for $\Lambda_{\rm plane}(n)$ and are postponed to Appendix~\ref{app: special functions}.


\subsection{Computer-assisted proof to estimate $\Lambda_{\plane}(n)$ and $M(k,l)$.}\label{ssec: computerassisted}
In the previous two subsections, we derived an exact expression \eqref{eq:final lambda plane expression} for $\Lambda_{\rm plane}(n)$ in terms of special functions, and an exact expression for $M(k,l)$, which bounds $\Lambda_{\rm Lawson}(n)$ above,  in terms of explicit integrals  by substituting \eqref{eqn: Lawson comp Flint vol} and \eqref{eqn: Lawson comp Flint perim} in \eqref{eqn: Lambda Lawson bounds}. Using the Arb library for arbitrary-precision ball arithmetic \cite{Joh17} within the FLINT mathematical computing library \cite{flint}, we obtain numerical approximations for both of these quantities with rigorous error estimates. As a consequence of these estimates we obtain the following:
\begin{proposition}\label{prop: computer}
    Let $n \in \mathbb{N}$, and for any $k,l \in \mathbb{N}$ with $k+l+2=n,$ let $M(k,l)$ be as in \eqref{eqn: def M k l}. Then the following holds.
    \begin{itemize}
        \item Suppose $n \in \{4, 2700\}$  is an even integer and let $k=\frac{n}{2}-1$. Then $M(k,k) < \Lambda_{\rm plane}(n)$. If $n \geq 8$ then additionally $M(k-1, k+1) < \Lambda_{\rm plane}(n).$
        \item Suppose $n \in \{5, 2700\}$ is an odd integer and let $k = (n-1)/2$. Then $M(k-1,k) < \Lambda_{\rm plane}(n)$.
    \end{itemize}
\end{proposition}
The source code (in C) used to prove Proposition~\ref{prop: computer} is included in the supplementary documents.

We now have:
\begin{proof}[Proof of Theorem~\ref{t:Simons-vs-lens}]
  Fix an integer $n \in \{4, 2700\}$. For any $l,k \in \mathbb{N}$ with $k+l+2 = n$, let $M(k,l)$ be as in \eqref{eqn: def M k l}. Since $\Lambda_{\Lawson}(n) \leq M(k,l)$ by \eqref{eqn: Lambda Lawson leq M}, Proposition~\ref{prop: computer} shows that $\Lambda_{\rm Lawson}(n) < \Lambda_{\rm plane}(n)$.
\end{proof}

A few comments are in order regarding Proposition~\ref{prop: computer}. First, we state Proposition~\ref{prop: computer} up to dimension $2700$ because of computational limitations. More specifically, between $2700$ and $2800$, our FLINT code is no longer able to estimate $\Lambda_{\rm plane}(n)$ with sufficient accuracy to draw any conclusions. 
We did not attempt to optimize the code or otherwise extend these numerics to higher dimensions. 

It might be reasonable to conjecture that the inequality  $\Lambda_{\rm Lawson}(n)< \Lambda_{\rm plane}(n)$ in all dimensions. It is not clear whether one should expect the specific competitor constructed in section~\ref{ss:Simons} to have $M(k,l) < \Lambda_{\rm plane }(n)$ in every dimension; the numerics show that $\Lambda_{\rm plane }(n) - M(k,l)$ is decreasing in $n$ (up to $n=2700$) for the choices of $k,l$ in Proposition~\ref{prop: computer}; see Figure~\ref{plot} below.

    We have numerically computed $\Lambda_{\rm plane}(n)$ for $n$ between $2$ and approximately $2750$, and it is increasing in $n$ within this range. Proving this fact could be interesting and possibly useful for constructing local minimizers with singular blowdowns in all dimensions $8$ and above.
 
Numerical approximations of $\Lambda_{\rm plane}(n)$ and $M(k,l)$ can be carried out in Mathematica as well, {if one expresses the integrals in terms of their respective special function representations}. We have chosen to work with FLINT because it uses interval arithmetic and thus provides rigorous error estimates. We have also carried out numerical approximations in Mathematica for the Simons cone case $k=l$ for even dimensions. While these numerics match those given by our FLINT code up to $n=36$, they become inaccurate in higher dimensions.

In the statement of Proposition~\ref{prop: computer}, we include the fact that both $M(k,k)$ and $M(k-1, k+1)$ are strictly less than $\Lambda_{\rm plane}(n)$ for even dimensions at least $8$  because this is of interest in the next section. We will show that in even dimensions, $M(k,k)$ (if $k$ is odd) or $M(k-1,k+1)$ (if $k$ is even) can in principle be computed by hand. Thus, at least in principle, one does not need a computer-assisted proof to establish Proposition~\ref{prop: computer} for even dimensions.

Below, in Table \ref{table:approx values}, we include a table showing numerical approximations of $\Lambda_{\rm plane}(n)$ and $M(k,l)$ for {the first few} values of $n \geq 8$. While our FLINT code approximates these numbers to over 20 digits, here we only include the expansion to 8 decimal places since for these values of $n$ it is sufficient to deduce the conclusion of Proposition~\ref{prop: computer}.

\begin{figure}[hbtp]
\centering\includegraphics[height=200pt]{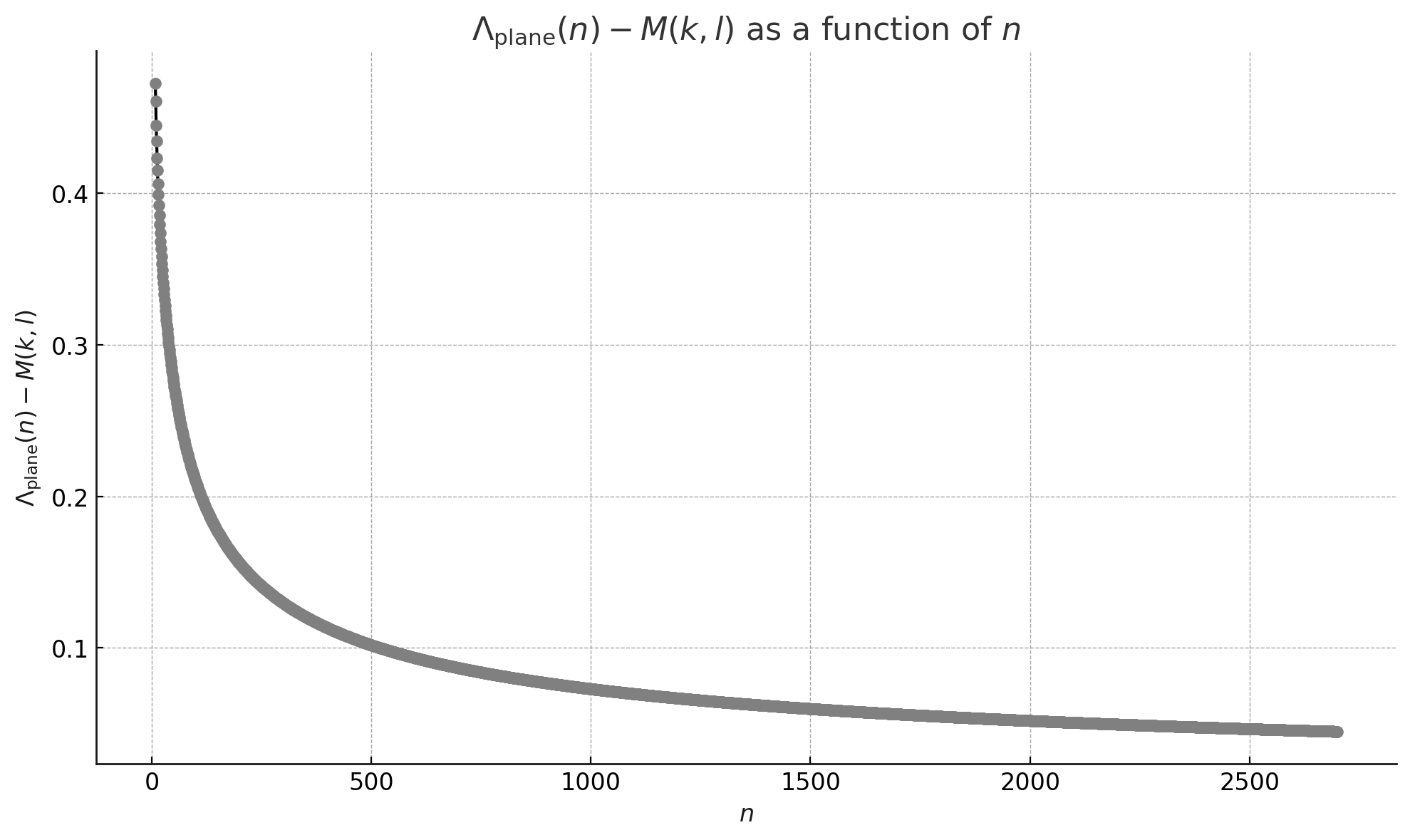}
\caption{\small A plot of $\Lambda_{\rm plane}(n) -M(k,l)$ as a function of $n$. Here $k=l=n/2-1$ when $n$ is even and $k=(n-3)/2, l=k+1$ when $n$ is odd.}
\label{plot}
\end{figure}

\begin{table}
\centering
\setlength{\arrayrulewidth}{0.3mm}
\setlength{\tabcolsep}{22pt}
\renewcommand{\arraystretch}{1.5}
\renewcommand{\arraystretch}{2.5}
\begin{center}
\begin{tabular}{||c c  c c c ||} 
 \hline
$n$ & $k$ & $l$ &$\approx \Lambda_{plane}(n)$  & $\approx M(k,l)$  \\ [0.5ex] 
 \hline
 8 & 3  & 3 & 7.29128238 & 6.81857964 \\ 
 \hline
 9 & 3 & 4 & 7.93735360& 7.47627954 \\ 
  \hline
10  & 4  & 4 & 8.55000228 & 8.10521276  \\ 
10  & 3 & 5 & --- &  8.09827171  \\
\hline
 11 & 4 &5 & 9.13366648 & 8.69902383 \\ 
\hline
12 & 5  & 5 &  9.69190314 & 9.26851974\\ 
 \hline
13 & 5 & 6 &  10.22761175 & 9.81252149 \\ 
 \hline
14 & 6  & 6 & 10.74319067 &  10.33685856 \\ 
14 & 5  & 7 & --- &  10.33488774 \\ 
 \hline
15 & 6 & 7 & 11.24064916 & 10.84138264 \\ 
  \hline
16 & 7 & 7 & 11.72168941 & 11.32970357\\ 
   [1ex] 
 \hline
\end{tabular}
\caption{Approximate values of $\Lambda_{\rm plane}(n)$ and $M_{k,l}$ for small $n$.}
\label{table:approx values}
\end{center}
\end{table}


\medskip

\subsection{By-hand computations in even dimensions and proof of Theorem \ref{t:Simons-vs-lens} for $n=8$}\label{ssec: by-hand}
While the computer-assisted proof of Proposition~\ref{prop: computer} includes rigorous error estimates, it is still interesting to understand the extent to which the proof of Proposition~\ref{prop: computer} and the associated computations can be carried out by hand. In this section, we demonstrate how this can be done for any even dimension $n=2k+2$. More specifically, we show that in this case it is possible to express $M(k,k)$ (if $k$ is odd) or $M(k-1,k+1)$ (if $k$ is even) entirely in terms of rational powers of integers and $\pi$. Moreover, in any dimension, it is possible to express $\Lambda_{\rm plane}(n)$ as a finite sum of products and quotients of polynomials, square roots, and inverse trigonometric functions evaluated at $z=1/4$. We carry out these details explicitly in $n=8$. From here, verifying Proposition~\ref{prop: computer} is a matter of estimating these numbers, {which again we carry out in $n=8$.}

\subsubsection{Further simplifying $\Lambda_{\rm plane}(n)$} To accompany our numerical approximation of $\Lambda_{\rm plane}(n)$, we {obtain here a closed form expression for $\Lambda_{\rm plane}(n)$.}  In this section, $n$ can be even or odd.

We first explain how the functions 
\begin{align}\label{eq:two hypers}
z\mapsto {}_2F_1\left(\mfrac{1}{2}, \, \mfrac{3-{n}}{2}; \, \mfrac{3}{2};\, z\right) \qquad\mbox{and}\qquad z\mapsto {}_2F_1\left(\mfrac{1}{2},\, \mfrac{1-{n}}{2}; \, \mfrac{3}{2};\, z\right) \qquad n\geq 8  
\end{align}
can be expressed as a finite sum of products and quotients of polynomials, square roots, and inverse trigonometric functions (and as observed above, when $n\geq 3$ is odd, it is a polynomial).
First, one may check by term by term series comparison that
\begin{align}\label{eq:special hyper}
 {}_2F_1\left(\mfrac{1}{2},\, \mfrac{1}{2};\,  \mfrac{3}{2};\, z\right) = \frac{\arcsin(\sqrt{z})}{\sqrt{z}}\qquad 0<z<1\,;
\end{align}
this formula is known in the literature, e.g.~ \cite[Equation 2.1.5]{andrewsaskeyroy}. 
{
Next, we recall the formula
\begin{align}\label{eq:gauss contiguous relation}
(c-b)\, {}_2F_1(a,b-1;c;  z) = z(1-z)\frac{d[{}_2F_1(a, b;c;z)]}{dz} - (b-c+az)\, {}_2F_1(a,b;c;z)\,,
\end{align}
which is one of Gauss's ``contiguous relations" \cite[Equation 2.5.9]{andrewsaskeyroy}. Equation \eqref{eq:gauss contiguous relation} implies that
\begin{align}\label{eq:gauss contiguous relation 2}
{}_2F_1\left(\mfrac{1}{2},\, b-1;\mfrac{3}{2};\,  z\right) = \frac{1}{\frac{3}{2}-b}\left[z(1-z)\frac{d[{}_2F_1(\mfrac{1}{2},\, b;\, \mfrac{3}{2}; \, z)]}{dz} - \big(b-\mfrac{3}{2}+\mfrac{z}{2}\big)\, {}_2F_1\left(\mfrac{1}{2},\, b;\, \mfrac{3}{2};\, z\right)\right],
\end{align}
allowing us to start with \eqref{eq:special hyper} and iteratively compute ${}_2F_1(1/2,m/2;3/2;z)$ for any $1\geq m\in \mathbb{Z}$. We thus obtain closed form expressions for each of the functions in \eqref{eq:two hypers}. We then evaluate these functions at $z=1/4$ as they appear in \eqref{eq:final lambda plane expression}.

} 

Additionally, $\Gamma(m/2)$ can be computed by hand for positive integers $m$ by using the property $\Gamma(x+1) = x\Gamma(x)$ together with the facts that $\Gamma(1)= 1$ and $\Gamma(1/2) = \sqrt{\pi}$. Therefore, a closed form expression for $\Lambda_{\plane}(n)$ for these values of $n$ can be obtained. 

As a concrete example, we describe the computation when $n=8$.

\begin{example}[Computation of $\Lambda_{\rm plane}(8)$]
When $n=8$, \eqref{eq:final lambda plane expression} reads
\begin{align}\label{eq:final lambda plane expression 8}
    \Lambda_{\plane}(8) 
    & = \omega_{7}^{1/8} \, \frac{7
\left(\frac{\sqrt{\pi} \Gamma\left(\frac{7}{2}\right)}{\Gamma\left(4\right)} - {}_2F_1\left(\frac{1}{2}, \frac{-{5}}{2} ;  \frac{3}{2} ;  \frac{1}{4}\right)\right) - \left(\frac{\sqrt{3}}{2}\right)^{7}}{\left(\frac{\sqrt{\pi} \Gamma \left(\frac{9}{2}\right)}{\Gamma\left(5\right)}-{}_2F_1\left(\frac{1}{2}, \frac{-{7}}{2};  \frac{3}{2}; \frac{1}{4}\right) \right)^{7/8}}\,.
\end{align}
Let us briefly summarize how to compute \eqref{eq:final lambda plane expression 8}. First,
\begin{align}\label{eq:half 8 computation}
   \omega_7 = \frac{16\pi^3}{105}\,,\quad \Gamma\left(\mfrac{7}{2}\right) = \frac{15\sqrt{\pi}}{8}\,,\quad \Gamma(4)=6\,, \quad \Gamma\left(\mfrac{9}{2}\right)=\frac{105\sqrt{\pi}}{16} \,, \quad \Gamma(5)=24\,.
\end{align}
Next, choosing $b=1/2$ in \eqref{eq:gauss contiguous relation 2}, using the formula \eqref{eq:special hyper}, and simplifying yields
\begin{align}\notag
    {}_2F_1\left(\mfrac{1}{2}, \mfrac{-1}{2};  \mfrac{3}{2}; z\right)&= \frac{1}{\frac{3}{2}-\frac{1}{2}}\left[z(1-z)\frac{d[{}_2F_1(\frac{1}{2},\frac{1}{2};\frac{3}{2};z)]}{dz} - \left(\mfrac{1}{2}-\mfrac{3}{2}+\mfrac{z}{2}\right){}_2F_1\left(\mfrac{1}{2},\mfrac{1}{2};\mfrac{3}{2};z\right)\right]\\ \notag
    &=  z(1-z)\left(\frac{1}{2z \sqrt{1-z}}- \frac{\arcsin(\sqrt{z})}{2z^{3/2}} \right) - \left(\frac{z}{2}-1 \right)\frac{\arcsin(\sqrt{z})}{\sqrt{z}} \\ \notag
    &= \frac{\sqrt{1-z}}{2}+\frac{\arcsin(\sqrt{z})}{2\sqrt{z}}\,.
\end{align}
Iterating this procedure three more times gives
\begin{align}\notag
{}_2F_1\left(\mfrac{1}{2}, \mfrac{-3}{2};  \mfrac{3}{2}; z\right)
&= \frac{1}{\frac{3}{2}-\frac{-1}{2}}\left[z(1-z)\frac{d[{}_2F_1(\frac{1}{2},\frac{-1}{2};\frac{3}{2},z)]}{dz} - \left(\mfrac{-1}{2}-\mfrac{3}{2}+\mfrac{z}{2}\right){}_2F_1\left(\mfrac{1}{2},\mfrac{-1}{2};\mfrac{3}{2} ; z\right)\right] \\ \notag
&=\sqrt{1-z}\left(-\mfrac{1}{4}z  +\mfrac{5}{8}\right)+\frac{3 \arcsin(\sqrt{z})}{8 \sqrt{z}}\,, \\ \notag
{}_2F_1\left(\mfrac{1}{2}, \mfrac{-5}{2};  \mfrac{3}{2}; z\right)
&= \frac{1}{\frac{3}{2}-\frac{-3}{2}}\left[z(1-z)\frac{d[{}_2F_1(\mfrac{1}{2},\mfrac{-3}{2};\mfrac{3}{2};z)]}{dz} - \left(\mfrac{-3}{2}-\mfrac{3}{2}+\mfrac{z}{2}\right){}_2F_1\left(\mfrac{1}{2},\mfrac{-3}{2};\mfrac{3}{2}; z\right)\right]\\ \label{eq:n=8 5formula} 
&= \sqrt{1-z}\left(\mfrac{1}{6}z^2  -\mfrac{13}{24} z +\mfrac{11}{16}\right)+\frac{5 \arcsin(\sqrt{z})}{16 \sqrt{z}}\,,\qquad \mbox{and} \\ \notag
{}_2F_1\left(\mfrac{1}{2}, \mfrac{-7}{2};  \mfrac{3}{2} ;  z\right) &= \frac{1}{\frac{3}{2}-\frac{-3}{2}}\left[z(1-z)\frac{d[{}_2F_1(\frac{1}{2},\frac{-3}{2};\frac{3}{2};z)]}{dz} - \left(\mfrac{-3}{2}-\mfrac{3}{2}+\mfrac{z}{2}\right){}_2F_1\left(\mfrac{1}{2},\mfrac{-3}{2};\mfrac{3}{2} ; z\right)\right] \\ \label{eq:n=8 7formula}
&=\sqrt{1-z}\left(-\mfrac{1}{8}z^3+\mfrac{25}{48}z^2  -\mfrac{163}{192} z +\mfrac{93}{128}\right)+\mfrac{35 \arcsin(\sqrt{z})}{128 \sqrt{z}}\,.
\end{align}
Since $\arcsin(\sqrt{1/4})=\pi/6$, we can insert \eqref{eq:half 8 computation}, \eqref{eq:n=8 5formula}, and \eqref{eq:n=8 7formula} into \eqref{eq:final lambda plane expression 8} and simplify. This yields 
\begin{equation}
    \label{eqn: Lambda Plane 8}
\Lambda_{\rm plane}(8) = 4\Big(\frac{2}{3}\Big)^{1/4}\pi^{3/8}\Big(-\frac{837\sqrt{3}}{35}+16\pi\Big)^{1/8}.
\end{equation}
\end{example}

In Table \ref{table:exact values}, we include exact values of $\Lambda_{\plane}(n)$ for some low values of $n$, which can be obtained by hand by carrying out the procedure laid out above, just as in the example above.


\begin{table}
\setlength{\arrayrulewidth}{0.3mm}
\setlength{\tabcolsep}{22pt}
\renewcommand{\arraystretch}{1.5}
\renewcommand{\arraystretch}{2.5}
\begin{center}
\begin{tabular}{||c c ||} 
 \hline
$n$ & $\Lambda_{\plane}(n)$  \\ [0.5ex] 
 \hline\hline
 8 & $4\left(\frac{2}{3}\right)^{1 / 4} \pi^{3 / 8}\left(-\frac{837 \sqrt{3}}{35}+16 \pi\right)^{1 / 8}$
 \\ 
 \hline
 9 & $\frac{3\left(\frac{53}{35}\right)^{1 / 9} 3^{2 / 3} \times 11^{2 / 9} \pi^{4 / 9}}{2 \times 2^{2 / 9}}$ 
 \\ 
 \hline
10 & $\frac{5^{4 / 5}(2 \pi)^{2 / 5}\left(-\frac{891 \sqrt{3}}{7}+80 \pi\right)^{1 / 10}}{3^{1 / 5}}$ 
\\ 
 \hline
 11 & $\frac{11^{10 / 11}\left(\frac{7199}{7}\right)^{1 / 11} \pi^{5 / 11}}{2 \times 2^{2 / 11} \times 3^{3 / 11}}$ 
 \\
  \hline
  12 & $\frac{2 \sqrt{2} 3^{3 / 4} \pi^{5 / 12}\left(-\frac{729 \sqrt{3}}{11}+40 \pi\right)^{1 / 12}}{5^{1 / 6}}$
  \\
    \hline
  13 & $\frac{13^{12 / 13}\left(\frac{407521}{385}\right)^{1 / 13} \pi^{6 / 13}}{2 \times 6^{3 / 13}}$ 
  \\
    \hline
  14 & $\frac{\sqrt{2} 7^{6 / 7} \pi^{3 / 7}\left(-\frac{12393 \sqrt{3}}{13}+560 \pi\right)^{1 / 14}}{3^{3 / 14} \times 5^{1 / 7}}$ 
  \\
    \hline
  15 & $\frac{3^{11 / 15} \times 5^{14 / 15}\left(\frac{464389}{1001}\right)^{1 / 15} \pi^{7 / 15}}{2 \times 2^{1 / 5}}$ 
  \\
    \hline
  16 &  $\frac{8 \times 2^{5 / 16} \pi^{7 / 16}\left(-\frac{277749 \sqrt{3}}{143}+1120 \pi\right)^{1 / 16}}{3^{3 / 16} \times 35^{1 / 8}}$
  \\
    \hline
\end{tabular}
\end{center}
\caption{Exact values of $\Lambda_{\rm plane}(n)$ for small $n$.}
\label{table:exact values}
\end{table}

\subsubsection{Evaluating $M(k,l)$ when $k$ and $l$ are odd}\label{r:simplified-Lambda-Lawson}
Suppose $k$ and $l$ are both positive odd integers. 
In this case, evaluating the expression for the competitor energy $M(k,l)$ in \eqref{eqn: Lambda Lawson bounds} 
 boils down to integrating polynomials, and thus can be expressed without using special functions. For simplicity, let us demonstrate this in the case of the Simons' cones, $k=l$, for $k$ odd. In this case, $\lambda=1$ and the constants from \eqref{eqn: r and h} and  \eqref{eqn: rho and d} satisfy $r=\rho$ and $h=d$.
Starting from \eqref{eqn: Lawson comp Flint perim}, then using the binomial expansion twice and integrating,
\begin{align}\notag   
\frac{\mathcal{H}^{n-1}( \partial E)}{2(k+1)^2\omega_{k+1}^2}& =   \int_{\partial \hat{E} \cap \{ u<v\}} u^k v^k \, d \mathcal{H}^{1} \\ \notag 
&= r  \int_1^{r-h}u^k \left( r^2 - (u +h )^2\right)^{\frac{k-1}{2}}\,du \\ \notag
 &= r  \int_{1}^{r-h}\sum_{j=0}^{(k-1)/2}r^{k-1 - 2j}\binom{\frac{k-1}{2}}{j}(-1)^j\sum_{\ell=0}^{2j}\binom{2j}{\ell}u^{k+\ell}h^{2j - \ell}\, du\\ \notag
 &=  \sum_{j=0}^{(k-1)/2}r^{k- 2j}\binom{\frac{k-1}{2}}{j}(-1)^j\sum_{\ell=0}^{2j}\binom{2j}{\ell}\frac{u^{k+\ell+1}}{k+\ell+1}h^{2j - \ell} \Big|_{u=1}^{u=r-h}\,.
 \end{align}
Similary, starting from \eqref{eqn: Lawson comp Flint vol}, the volume term can be computed as
 \begin{align}\notag
\frac{|E|}{\omega_{k+1}^2} 
 &= 1 + 2(k+1) \int_1^{r-h} u^k \left( r^2  - (u +h )^2\right)^{\frac{k+1}{2}} \, du \\ \notag
 &= 1 + {2(k+1)}\sum_{j=0}^{(k+1)/2}r^{k+1 - 2j}\binom{\frac{k+1}{2}}{j}(-1)^j\sum_{\ell=0}^{2j}\binom{2j}{\ell}\frac{u^{k+\ell+1}}{k+\ell+1}h^{2j - \ell} \Big|_{u=1}^{u=r-h}\,.
\end{align}
Of course, the analogous computations hold when $k$ and $l$ are distinct and both odd. Thus, in these cases,
the expression for $M(k,l)$ in \eqref{eqn: Lambda Lawson bounds} can be expressed in terms of rational powers of integers and $\pi$. For example, when $n=8$ and $k=l=3$,
\[
M(3,3) = \frac{2\left(\frac{2}{105}\right)^{1 / 8}\left(-699776+494843 \sqrt{2}-404096 \sqrt{3}+285740 \sqrt{6}\right) \sqrt{\pi}}{\left(913063-645632 \sqrt{2}+527138 \sqrt{3}-372736 \sqrt{6}\right)^{7 / 8}}.
\]


\section{Capillarity problems inside cones}\label{sec:cap}

For an open $K\subset \R^n$ and $\beta \in (-1,1)$, consider the problem

\begin{align}\label{eq:cap inside Lawson}
   \Lambda(K,\beta):= \inf\{P(E;K) - \beta P(E;\partial K) : E \subset K,\, |E|=1/2  \}\,.
\end{align}

\begin{remark}[On $\Lambda(K,\beta)$]
    We have chosen this notation because when $\beta=1/2$ and $K$ is a half-space, $\Lambda(K,\beta) = \Lambda_{\rm plane}(n)/2$. This follows from the characterization of the upper half of the first chamber of the lens cluster as the unique minimizer in a classical capillarity problem \cite[Theorem 3.4]{BroNov24} and the local minimality of the lens. While it is plausible to expect the connection between locally minimizing $(1,2)$-clusters and minimizers for capillarity problems to extend to cones other than planes which possess some symmetry analogous to reflection over a hyperplane, such as invariance under the map $(x_1,\dots,x_{k+1},y_{1},\dots,y_{k+1})\mapsto (y_1,\dots,y_{k+1},x_{1},\dots,x_{k+1})$ for the cones $C_{k,k}$, this is still open. Note that in fact it is not known whether the singular minimizer constructed in Theorem \ref{t:higher-dim} blows down $C_{k,k}$ even if starting with $C_{k,k}$ in the penalized problem \eqref{e:localized-min-problem}. The only relationship follows readily from the definitions is for example
    \begin{align}\notag
        2\Lambda(\{x_1^2 + \dots + x_{k+1}^2 < y_1^2 + \dots + y_{k+1}^2\},1/2) \geq \Lambda(C_{k,k})\,,
    \end{align}
    which follows from the fact that the admissible class for \eqref{eq:cap inside Lawson} is contained in that of $\Lambda(C_{k,k})$ via reflection over $C_{k,k}$.
\end{remark}

\begin{theorem}\label{t:capillary}
    Suppose that $K$ is an open cone, $\partial K$ has an isolated singularity and either $\beta\in (-1,0)$, or $\beta \in [0,1)$ and $K$ is a perimeter minimizer. If in addition
    \begin{align}\label{eq:suff cond for existence of cap}
    \Lambda(K,\beta)<\Lambda(\mathbb{R}^{n-1}\times (0,\infty),\beta) \,,   
    \end{align}
    then there exists a minimizer to the problem \eqref{eq:cap inside Lawson}. 
\end{theorem}
\begin{remark}[On the assumption that $\partial K$ is area-minimizing if $\beta \in [0,1)$]
    We only use this assumption in one place, namely to obtain a short proof of the bound 
    \begin{align}\notag
        P(E) \lesssim P(E;K) - \beta P(E;\partial K) \,. \end{align}
    If this bound can be generalized to other open cones $K$ with isolated singularity, then the existence result in these cases follow as well. We have not attempted to optimize these assumptions here, partially because the cases in which we are interested are coming from locally minimizing $(1,2)$ clusters involve area-minimizing cones.
\end{remark}

\begin{proof}[Proof of Theorem \ref{t:capillary}]

The proof involves a concentration compactness argument which is similar to but simpler than Theorem \ref{t:higher-dim}, due to the fact that the reference cone is fixed in the capillarity problem. Therefore, we provide summaries in lieu of complete details.

\medskip

\noindent{\it Penalized problem and energy bounds}: Fix a sequence of radii $R_k\uparrow\infty$. With $g_{k}$ and $\cG_{k}$ as in Section \ref{s:compactness} we consider the energy
\begin{align}\notag
    \mathcal{F}_k(E;B_{4R_k}):= P(E;B_{4R_k} \cap K) - \beta P(E; B_{4R_k} \cap \partial K) + \cG_{k}(E)\,.
\end{align}
and the corresponding variational problem
\begin{align}\label{eq:penalized cap problem}
    \nu_k = \inf\{\mathcal{F}_k(E;B_{4R_k}) : |E|=1\,, E \subset K \cap B_{3R_k} \}\,.
\end{align}
When $\beta \leq 0$, existence follows by quoting the lower-semicontinuity result \cite[Proposition 19.1]{Mag12}. When $\beta>0$, since $\partial K$ only has an isolated singularity and is smooth elsewhere, straightforward modifications to the classical lower-semicontinuity for capillarity energies \cite[Proposition 19.3]{Mag12} yield the existence of a minimizer $E_k$ for $\nu_k$. Note that since the infimum $\Lambda(K,\beta)$ can be achieved by bounded sets (via a truncation argument) and $\mathcal{F}_k(E;B_{4R_k})=P(E;K) - \beta P(E;\partial K)$ if $E\subset B_{\sqrt{R_k}}$, via an energy comparison for $k$ sufficiently large we have the upper bound
\begin{align}\label{eq:lawson beats plane capillarity}
  \limsup_{k\to \infty}\mathcal{F}_k(E_k;B_{4R_k}) \leq \Lambda(K,\beta) \,,
\end{align}
which is the capillarity analogue of the admissibility of $\bar\X$ for $\mu_k$ in Section \ref{subsec:comp for lensshaped concentration}.

Towards obtaining uniform perimeter bounds for $E_k$, we first note that if $\beta<0$, then due to the fact that $E_k \subset B_{3R_k}$ we have the inequality
\begin{align}\label{eq:per bound when beta positive}
   |\beta|P(E_k) < P(E_k;K) - \beta P(E_k;\partial K) \leq \mathcal{F}_k(E;B_{4R_k}) 
\end{align}
On the other hand, if $\beta\geq 0$ and we are assuming perimeter minimality of $K$, then for any $E\subset K \cap B_{3R_k}$,
\begin{align}\notag
    P(K;B_{4R_k}) \leq P(K \setminus E;B_{4R_k})= P(E;K) + P(K;B_{4R_k}) - P(E;\partial K) \,.
\end{align}
Upon subtracting $P(K;B_{4R_k})-P(E;\partial K \cap B_{4R_k})$ from both sides, we obtain
\begin{align}\notag
    P(E;\partial K \cap B_{4R_k}) \leq P(E;K \cap B_{4R_k})\,,
\end{align}
which implies that for any $E\subset K \cap B_{3R_k}$,
\begin{align*}
(1-\beta)P(E)/2&= (1-\beta)P(E;K)/2+(1-\beta)P(E;\partial K \cap B_{4R_k})/2 \\
&\leq \mathcal{F}_k(E;B_{4R_k})  \,. 
\end{align*}
Thus by comparison arguments, replacing $E_k$ and $E_k \cap B_\rho(y)$ respectively with balls (cf.~ Section \ref{s:setup}), we obtain the respective $k$-independent energy bounds
\begin{align}\label{eq:cap total perimeter bound}
    P(E_k)/C_1 \leq \mathcal{F}_k(E_k;B_{4R_k}) &\leq C_1 \\ \label{eq:cap local perimeter bound}
    P(E_k;B_\rho(y)) &\leq C_2\rho^{n-1} + \frac{2\rho \max\{1,\omega_n \rho^n\}}{\sqrt{R}}\qquad \forall y\in \overline{K}
\end{align}
for some $C_1>1$ and $C_2>0$. As a consequence of \eqref{eq:cap total perimeter bound}, the penalization term again forces confinement of $E$ to balls $B_{\sqrt{R_k}+s}$ in a measure sense (cf.~ \eqref{eqn: measure confinement}): there is a constant $C_3>0$ (independent of $k$) such that for any $s>0$,
\begin{equation}\label{eq:cap measure confinement}
   |E \setminus B_{\sqrt{R_k}+s}|  \leq \frac{ C_3\sqrt{R_k}}{s}\,.  
\end{equation}

\medskip

\noindent{\it Nucleation, volume fixing variations, and almost-minimality}: As in Theorem \ref{t:higher-dim}, the global perimeter bound \eqref{eq:cap total perimeter bound} allows for the application of the nucleation lemma \cite[Lemma 29.10]{Mag12} to identify a finite number, say $N_k$, of balls $B_1(x_{k,i})$ (with $N_k \leq N_0$), each containing at least a fixed amount of volume of $E_k$ and whose union contains most of the volume of $E_k$ (again independently of $k$). By choosing the parameters in the same fashion as in Step 1 of Lemma \ref{lem: density estimates}, we obtain the analogue of \eqref{eqn: g bound} for the capillarity problem, which guarantees that 
\begin{align}\label{eq:g bound for cap}
 |x_{k,i}|\leq \mathrm{O}(\sqrt{R_k})\qquad \mbox{and}\qquad   g_{k}(|x_{k,i}|) \leq C_4 \qquad \forall k,\,i
\end{align}
for some $C_4$ independent of $k$. 

Now we are going to construct uniform volume fixing variations and use it to derive a ``one-sided" minimality inequality that will yield lower density estimates (such a property was used in \eqref{eq:pot estimate} in Lemma \ref{lem: density estimates}). First, note that since $\partial K$ is a regular cone, then for any $x_{k,i}$ escaping to infinity as $k\to \infty$, $K \cap B_1(x_{k,i})$ is approximately the intersection $B_1(x_{k,i})$ with the halfspace determined by the nearby tangent hyperplanes of $K$. By the uniform lower bound on $|B_1(x_{k,i})\cap E_k|$ from the nucleation lemma and the fact that $E_k\subset K$, the measures of $|K \cap B_1(x_{k,i})|$ are uniformly bounded from below in $i$ and $k$. Combining all these facts, we see that for any $i$, up to subsequence, not relabelled, there is an open set $A_i$ such that $K \cap B_1(x_{k,i})$ converges to some set $A_i \subset \bar{K}$ as $k\to+\infty$ in the Hausdorff sense. By the uniform perimeter bound \eqref{eq:cap total perimeter bound} and the lower bounds on $|B_1(x_{k,i}) \cap E_k|$, for each $i$ there is a limiting set of finite perimeter $E_{\infty,i}\subset A_i$ for $B_1(x_{k,i}) \cap E_k$. Thus by combining the usual volume fixing variations along a subsequence with the uniform bound \eqref{eq:g bound for cap} for $g_{k}$ as in Step 2 of the proof of Lemma \ref{lem: density estimates}, we may obtain volume fixing variations with two important properties: first, the variations are compactly supported away from $\partial K$ (so they do not influence the energy along $\partial K$) and second, for each volume increment $\sigma$, the variation in the energy $\mathcal{F}_k$ is at most $C_5|\sigma|$ for some $C_5>0$ independent of $k$. 

Next, by combining the minimality of $E_k$ for $\nu_k$ with the volume fixing variations, we claim there are $r_0>0$ and $\Lambda>0$, both independent of $k$, such that $E_k$ satisfies the ``one-sided" almost-minimizing condition
\begin{align}\notag 
    &P(E_k;K \cap B_{4R_k}) - \beta P(E_k;\partial K\cap B_{4R_k}) \\ \label{eq:almost cap minimality}
    &\qquad \qquad \leq P(E;K\cap B_{4R_k})- \beta P(E;\partial K\cap B_{4R_k}) + \Lambda\big||E_k| -|E|\big|
\end{align}
whenever $E\subset E_k$ and $E_k \setminus E \subset B_{r_0}(y)\subset B_{3R_k}$. Indeed, choosing $r_0$ small enough so that $\omega_nr_0^n$ of volume can be added using the volume fixing variations, this follows by a direct comparison argument using the minimality of $E_k$ for $\nu_k$. 

\medskip

\noindent{\it Lower density estimates}: We claim now that there exists $c_0$ and $\rho_0$, both independent of $k$, such that if $y \in 
B_1^c \cap  \spt\, \mathcal{H}^{n-1}\mres (\partial^* E_k \cap K)$ and $0<\rho<\rho_0$, then
\begin{align}\label{eq:lower density cap}
     |E_k \cap B_{\rho}(y)| \geq c_0 \rho^n\,.
\end{align}
If $y$ is at a fixed distance $d$ from $\partial K$, then such an estimate for radii $\rho<\min\{d,\rho_1\}$ holds by the usual comparison argument with $E_k \setminus B_\rho(y)$ \cite[Chapter 21]{Mag12}. On the other hand, for $y$ nearby $\partial K$ (and away from the singularity by assumption), we instead transform the isotropic capillarity problem in $K \cap B_\rho(y)$ to an anisotropic one in $H \cap A$, where, for some boundary flattening diffeomorphism $\Phi$ and halfspace $H$, we have $\Phi(K \cap B_\rho(y)) = H \cap A$. We will appeal to lower density results for anisotropic capillarity functionals in halfspaces from \cite{DePMag15}, as we now describe. After this boundary flattening, the one-sided minimality condition \eqref{eq:almost cap minimality} holds for $\Phi(E_k)$ in $A \cap H$ with respect to an anisotropic energy 
\begin{align}\notag
   \int_{H \cap \partial^* \Phi(E_k)} \Psi(\nu_{\Phi(E_k)}) \,d\mathcal{H}^{n-1} - \beta \int_{\partial H \cap \partial^* \Phi(E_k)} \Psi(\nu_{\Phi(E_k)}) \,d\mathcal{H}^{n-1}
\end{align}
and new constants $\Lambda'$, $r_0'$ that depend on the diffeomorphism $\Phi$, in place of $\Lambda$, $r_0$ respectively; see e.g.~ \cite[Lemma 2.17, Proof of Theorem 1.2]{DePMag15} for details on how elliptic anisotropies and almost minimizers behave under such transformations. Furthermore, we may further transform this energy by eliminating the adhesion term along $\partial H$ and changing the anisotropy \cite[Lemma 6.1]{DePMag15}, so that the one-sided minimality condition holds with respect to an energy of the form
\begin{align}\notag
   \int_{H \cap \partial^* \Phi(E_k)} \tilde{\Psi}(\nu_{\Phi(E_k)}) \,d\mathcal{H}^{n-1}\,.
\end{align}
This is done so that we can repeat the argument in \cite[Proof of (2.48) in Lemma 2.8]{DePMag15}, which, similar to the interior case, is based on comparison with $\Phi(E_k) \setminus B_{\rho'}(\Phi(y))$ and isoperimetry, and requires an anisotropic energy with no adhesion term. We thus obtain local lower density bounds for $\Phi(E_k)$ at $\Phi(y)$. These can be translated to $k$-independent lower density bounds for $E_k$ at $y$ through the diffeomorphism $\Phi^{-1}$. It remains to show that such estimates can be made independently of $y \in 
B_1^c \cap  \spt\, \mathcal{H}^{n-1}\mres (\partial^* E_k \cap K)$ and $0<\rho<\rho_0$. This is where we use the fact that $\partial K$ is a smooth cone away from $0$, since it implies that the diffeomorphisms used to straighten $\partial K$ at a fixed order $1$ scale converge to the identity as $|y|\to \infty$. The estimates in \cite[Lemma 2.8, Lemma 6.1]{DePMag15} depend only on suitable ellipticity/smoothness type properties of $\Psi$, $\tilde{\Psi}$ and $\big|1-|\beta|\big|$ (see the paragraph below \cite[Equation (6.7)]{DePMag15}). Since $\Phi$ approximates the identity map in every $C^k$ if $|y|$ is large and $|\beta|<1$, the desired $y$-independence follows.

\medskip

\noindent{\it Conclusion}: The conclusion of the proof is a simplified version of the arguments at the end of the proof of Theorem \ref{t:higher-dim}, since we do not need to deal with blowdown of an exterior interface. We first apply the nucleation lemma, with parameter $\varepsilon$ (bounding the volume of $E_k$ appearing outside the balls) chosen small enough so that $\omega_n\rho_0^n>\e$. Combining this with the density estimate \eqref{eq:lower density cap}, for each $k$ we then deduce that each $E_k$ is contained in the union of finitely many balls $B_{R_0}(x_{k,i})$, $1\leq i \leq I$. Let us call $v_{k,i}=|E_k \cap B_{R_0}(x_{k,i})|$, with
\begin{align}\label{eq:equiv class 2}
    \dist(x_{k,i},x_{k,j})\to \infty\qquad \mbox{if $i\neq j$.}
\end{align}
Up to a subsequence we may assume by the lower density bound that $v_{k,i}\to v_i$ for some $v_i>0$ for each $i$. We claim that if $|x_{k,i}|\to \infty$ for any $i$, then 
\begin{align}\label{eq:goes bound}
\mathcal{F}_k(E_k;B_{R_0}(x_{k,i})) = 
v_i^{(n-1)/n}\Lambda(\mathbb{R}^{n-1}\times (0,\infty),\beta) + v_i g_{k} (|x_{k,i}|)  + {\rm o}_k(1)\,.
\end{align}
Indeed, the upper bound for the energy by the right hand side follows by a comparison argument and the fact that ${\rm Lip}(g_{k})\to 0$, and the lower bound follows by straightening the boundary and using the fact that the resulting anisotropies are arbitrarily close to the identity together with the minimality of the lens. On the other hand, by \eqref{eq:equiv class 2}, there is at most one $i$ for which $|x_{k,i}|$ remains bounded, which without loss of generality we may assume is $i=1$ if it exists (otherwise we simply start counting from $i=2$ and call $v_1=0$), for which we have
\begin{align}\label{eq:stays bound}
    \mathcal{F}_k(E_k;B_{R_0}(x_{1,k})) = 
v_1^{(n-1)/n} \Lambda(K,\beta)+{\rm o}_k(1)\,.
\end{align}
The lower bound of the left hand side by the right hand side is a consequence of the definition \eqref{eq:cap inside Lawson}, and the upper bound follows by energy comparison, analogously to \eqref{eq:lawson beats plane capillarity}. By using in order \eqref{eq:lawson beats plane capillarity}, \eqref{eq:goes bound}-\eqref{eq:stays bound}, \eqref{eq:suff cond for existence of cap}, and the concavity of $v\mapsto v^{(n-1)/n}$, we conclude that 
\begin{align}\notag
   \Lambda(K,\beta)&\geq \limsup_{k\to \infty}\mathcal{F}_k(E_k;B_{4R_k})\\ \notag
   &\geq \limsup_{k\to \infty}\sum_{i\geq 2} v_i^{(n-1)/n}\Lambda(\mathbb{R}^{n-1}\times (0,\infty),\beta) + v_1^{(n-1)/n} \Lambda(K,\beta)\\ \notag
   &\geq \limsup_{k\to \infty}\sum_{i} v_i^{(n-1)/n}\Lambda(K,\beta) \geq \left( \sum_i v_i\right)^{(n-1)/n}\Lambda(K,\beta) = \Lambda(K,\beta)\,.
\end{align}
But since each $v_i$ is positive for $i\geq 2$ and \eqref{eq:suff cond for existence of cap} holds, equality can only hold in the first inequality in the final line above if no $x_{k,i}$ can diverge to infinity, which means there is one large ball $B_{R_1}$ containing every $E_k$. By compactness, minimality of $E_k$ for $\nu_k$, and lower-semicontinuity, we obtain limiting $E$ minimizing \eqref{eq:cap inside Lawson}. 
\end{proof}

\begin{remark}
    Note that in the setting of Theorem \ref{t:higher-dim}, despite the validity of the strict inequality $\Lambda(\partial K) < \Lambda_{\plane}(n)$, we cannot conclude that none of the $x_{k,i}$ can diverge to infinity, since we are missing both the analogue of \eqref{eq:stays bound} for both the concentration that might remain bounded, and for any concentration that might diverge to infinity in a singular-shaped form (i.e. not lens-shaped). On the other hand, if we knew $\partial K$ had the lowest density among all area-minimizing hypercones in the given dimension, we could guarantee these missing properties; see the discussion following Theorem \ref{t:higher-dim}.
\end{remark}

\subsection*{Methods} The large language model ChatGPT was used in the writing of the C code, which is attached in the supplementary files. The authors have reviewed and edited the code, and accept responsibility for its content.

\subsection*{Acknowledgments} The authors thank Dominik Stantejsky for providing a numerical estimation of $\Lambda(C_{3,3})$. M.N. is grateful to Sean Carney for a valuable discussion regarding numerical computation of $\Lambda_{\plane}$ and $\Lambda_{\Lawson}$. The authors are grateful to Marcus Greenwood for providing assistance with producing the FLINT code for the numerical computations of $\Lambda_{\plane}(n)$ and $M(k,l)$. A.S. is grateful for the generous support of Dr.~ Max R\"ossler, the Walter Haefner Foundation and the ETH Z\"urich Foundation. This research was partially conducted during the period A.S. served as a Clay Research Fellow.
RN is grateful for the support of NSF grants DMS-2340195, DMS-2155054, and DMS-2342349. LB is supported by an NSERC Discovery grant. Part of this work was done when R.N., M.N., and A.S. were visiting L.B. at McMaster University.
\appendix

\section{Monotonicity identity}

The next lemma is the usual monotonicity identity but without the assumption that the distributional mean curvature is absolutely continuous with respect to the weight measure. This entails essentially no change in the proof, which we include for convenience.

\begin{lemma}[Monotonicity]\label{lemma:monotonicity}
    Let $V=\underline{v}(M,\theta)$ be a rectifiable $m$-varifold in an open set $U \subset \mathbb{R}^{m+\ell}$ with locally bounded first variation $\delta V$, so that $\delta V$ is a Radon measure and, setting $\nu = D_{|\delta V|}\delta V$, 
\begin{align}\label{eq:first variation}
    \int_{U} \Div_{M} X \, d\mu_V = \int_U X \cdot \nu\,d|\delta V|\qquad \forall X \in C_c^\infty(U;\mathbb{R}^{m+\ell}),
\end{align}
where $\mu_V = \theta \mathcal{H}^m \mres M$; cf. \cite[Chapter 8]{Sim83}. Then, for any $x\in U$ and $0<r<s<\dist(x,\partial U)$, 
\begin{align}\notag
\frac{\mu_V(B_s(x))}{s^m} - \frac{\mu_V(B_r(x))}{r^m} &= \int_{B_s(x) \setminus B_r(x)}\frac{|(y-x)^\perp|^2}{|y-x|^{m+2}}\,d\mu_V \\ \label{eq:monotonicity identity}
&\qquad + \int_{B_s(x)}\bigg[\frac{1}{s^m}-\frac{1}{\max\{r,|y-x|\}^m}\bigg]\frac{(y-x)\cdot \nu}{m}\,d|\delta V|\,.
\end{align}
\end{lemma}

\begin{proof}
    For ease of notation, let us take $x=0$. For $\rho \in (r,s)$, let $X(y) = \varphi(|y|/\rho)y$, where $\varphi \in C_c^1([0,1);[0,1])$. Denoting $\hat y^\perp=\mathbf{proj}_{T_yM^\perp}(y/|y|)$, direct computation yields
\begin{align*}
    \Div_M X = \varphi'\bigg(\frac{|y|}{\rho}\bigg)\frac{|y|}{\rho}\big( 1-|\hat y^\perp|^2\big) + m \varphi\bigg(\frac{|y|}{\rho}\bigg)\,.
\end{align*}
Testing \eqref{eq:first variation} with this $X$ and isolating the term with $-|y|\varphi'|\hat y^\perp|^2/\rho$ on one side, we obtain
\begin{align}\notag
    - \int_U \varphi'\bigg(\frac{|y|}{\rho}\bigg)\frac{|y|}{\rho}|\hat y^\perp|^2 \,d\mu_V &= \int_U \varphi\bigg(\frac{|y|}{\rho}\bigg)y \cdot \nu\, d|\delta V| - \int_U \bigg[\frac{\varphi'(|y|/\rho)|y|}{\rho} + m\varphi(|y|/\rho)\bigg]\,d\mu_V\\ \label{eq:radial test 1}
    &= \int_U \varphi\bigg(\frac{|y|}{\rho}\bigg)y \cdot \nu\, d|\delta V| + \int_U \rho^{m+1}\frac{d}{d\rho}\bigg[\frac{\varphi(|y|/\rho)}{\rho^m}\bigg]\,d\mu_V\,.
\end{align}
We multiply both sides of \eqref{eq:radial test 1} by $\rho^{-m-1}$ and integrate in $\rho$ from $r$ to $s$, which gives
\begin{align}\notag
    -\int_r^s \int_U &\varphi'\bigg(\frac{|y|}{\rho}\bigg)\frac{|y|}{\rho^{m+2}}|\hat y^\perp|^2 \,d\mu_V \, d\rho \\ \notag
    &= \int_r^s\int_U \varphi\bigg(\frac{|y|}{\rho}\bigg)\frac{y \cdot \nu}{\rho^{m+1}}\, d|\delta V|\,d\rho + \int_r^s\int_U \frac{d}{d\rho}\bigg[\frac{\varphi(|y|/\rho)}{\rho^m}\bigg]\,d\mu_V\,d\rho \\ \label{eq:radial test 2}
    &= \int_Uy \cdot \nu\int_r^s \varphi\bigg(\frac{|y|}{\rho}\bigg)\frac{1}{\rho^{m+1}}\,d\rho\, d|\delta V| + \int_U  \bigg[\frac{\varphi(|y|/s)}{s^m}-\frac{\varphi(|y|/r)}{r^m}\bigg]\,d\mu_V\,,
\end{align}
where in the last line we have used Fubini's theorem and the fundamental theorem of calculus. The left hand side of \eqref{eq:radial test 2} we now rewrite using Fubini and then integration by parts in $\rho$:
\begin{align}\notag
    -\int_r^s \int_U \varphi'\bigg(\frac{|y|}{\rho}\bigg)\frac{|y|}{\rho^{m+2}}|\hat y^\perp|^2 \,d\mu_V \, d\rho&= \int_U |\hat y^\perp|^2 \int_r^s \bigg[- \varphi'\bigg(\frac{|y|}{\rho}\bigg) \frac{|y|}{\rho^2}\bigg]\frac{1}{\rho^m}\,d\rho \,d\mu_V \\ \label{eq:radial test 3}
    &= \int_U |\hat y^\perp|^2 \bigg[\frac{\varphi(|y|/s)}{s^m}-\frac{\varphi(|y|/r)}{r^m} + \int_r^s \frac{m\varphi(|y|/\rho)}{\rho^{m+1}}\,d\rho\bigg]d \mu_V.
\end{align}
Inserting \eqref{eq:radial test 3} back into the left hand side of \eqref{eq:radial test 2}, we arrive at 
\begin{align}\notag
    \int_U |\hat y^\perp|^2 &\bigg[\frac{\varphi(|y|/s)}{s^m}-\frac{\varphi(|y|/r)}{r^m} + \int_r^s \frac{m\varphi(|y|/\rho)}{\rho^{m+1}}\,d\rho\bigg]d \mu_V \\ \label{eq:radial test 4}
    &= \int_Uy \cdot \nu\int_r^s \varphi\bigg(\frac{|y|}{\rho}\bigg)\frac{1}{\rho^{m+1}}\,d\rho\, d|\delta V| + \int_U \bigg[\frac{\varphi(|y|/s)}{s^m}-\frac{\varphi(|y|/r)}{r^m}\bigg]\,d\mu_V\,.
\end{align}
We now let $\varphi \to \mathbf{1}_{[0,1]}$ and apply the dominated convergence theorem to \eqref{eq:radial test 4} to conclude that
\begin{align}\notag
    \int_U |\hat y^\perp|^2 &\bigg[\frac{\mathbf{1}_{B_s}(y)}{s^m}-\frac{\mathbf{1}_{B_r}(y)}{r^m} + \int_r^s \frac{m\mathbf{1}_{B_\rho}(y)}{\rho^{m+1}}\,d\rho\bigg]d \mu_V \\ \label{eq:radial test 5}
    &= \int_Uy \cdot \nu\int_r^s \mathbf{1}_{B_\rho}(y)\frac{1}{\rho^{m+1}}\,d\rho\, d|\delta V| + \int_U \bigg[\frac{\mathbf{1}_{B_s}(y)}{s^m}-\frac{\mathbf{1}_{B_r}(y)}{r^m}\bigg]\,d\mu_V \,.
\end{align}
To simplify both sides of \eqref{eq:radial test 5}, we use the equivalence
$$
\mathbf{1}_{(r,s)}(\rho)\mathbf{1}_{(|y|,\infty)}(\rho)=\mathbf{1}_{(r,s) \cap (|y|,\infty)}(\rho) = \mathbf{1}_{B_s}(y)\mathbf{1}_{(\max\{r,|y|\},s)}(\rho)
$$
to compute
\begin{align}\notag
  \int_r^s \frac{\mathbf{1}_{B_\rho}(y)}{\rho^{m+1}}\,d\rho = \int \frac{\mathbf{1}_{(r,s)}(\rho)\mathbf{1}_{(|y|,\infty)}(\rho)}{\rho^{m+1}}\,d\rho &= \int \frac{\mathbf{1}_{B_s}(y)\mathbf{1}_{(\max\{r,|y|\},s)}(\rho)}{\rho^{m+1}}\,d\rho \\ \label{eq:radial test 6}
  &= -\frac{\mathbf{1}_{B_s}(y)}{ms^m} + \frac{\mathbf{1}_{B_s}(y)}{m\max\{r,|y|\}^m}\,.
\end{align}
If we substitute \eqref{eq:radial test 6} into \eqref{eq:radial test 5}, we have
\begin{align}\notag
    \int_U &|\hat y^\perp|^2 \bigg[ \frac{\mathbf{1}_{B_s}(y)}{\max\{r,|y|\}^m} - \frac{\mathbf{1}_{B_r}(y)}{r^m}\bigg]d \mu_V \\ \notag
    &= \int_Uy \cdot \nu\bigg[-\frac{\mathbf{1}_{B_s}(y)}{ms^m} + \frac{\mathbf{1}_{B_s}(y)}{m\max\{r,|y|\}^m}\bigg] d|\delta V| + \int_U \bigg[\frac{\mathbf{1}_{B_s}(y)}{s^m}-\frac{\mathbf{1}_{B_r}(y)}{r^m}\bigg]\,d\mu_V \\ \label{eq:radial test 7}
    &= \int_{B_s(0)} \frac{y \cdot \nu}{m}\bigg[-\frac{1}{s^m} + \frac{1}{\max\{r,|y|\}^m}\bigg] d|\delta V| + \frac{\mu_V(B_s)}{s^m} - \frac{\mu_V(B_r)}{r^m}\,.
\end{align}
Since the integrand on the left hand side simplifies to $\mathbf{1}_{B_s\setminus B_r}(y)|\hat y^\perp|^2 / |y|^m$, we have proved \eqref{eq:monotonicity identity}.
\end{proof}

\section{Expression for \texorpdfstring{$M(k,l)$}{Mkl} in terms of special functions}\label{app: special functions}
\begin{proof}[Proof of \eqref{eqn: lawson competitor volume}]
We write the first integral in  \eqref{eqn: Lawson comp Flint vol} as
\begin{align} \label{eqn: term I and term II}
    \int_0^{\rho-d} u^k \big(\rho^2 -(u+d)^2\big)^{\frac{l+1}{2}} \,du -
    \int_0^\lambda u^k \big(\rho^2 -(u+d)^2\big)^{\frac{l+1}{2}} \,du = I - II\,.
\end{align}
For term $I$ in \eqref{eqn: term I and term II}, we factorize and use the change of variable $(\rho-d)t = u$ to find
\begin{align*}
    I & = \int_0^{\rho-d} u^k \big(\rho - u-d\big)^{\frac{l+1}{2}} \big(\rho + u + d\big)^{\frac{l+1}{2}} \ ,du  \\
    & = {(\rho^2 - d^2)^{\frac{l+1}{2}}} \int_0^{\rho-d} u^k \Big(1 - \mfrac{u}{\rho-d} \Big)^{\frac{l+1}{2}} \Big(1 + \mfrac{u}{\rho+d} \Big)^{\frac{l+1}{2}} \, du\\
     & = {(\rho^2 - d^2)^{\frac{l+1}{2}}} (\rho -d)^{k+1}\int_0^1 t^k (1-t)^{\frac{l+1}{2}} \Big( 1 + \mfrac{\rho - d}{\rho+ d} \, t \Big)^{\frac{l+1}{2}}\, dt
\end{align*}
By the Euler integral formula (see, e.g. \cite[\S 2.1.3]{bateman_2023_cnd32-h9x80})
\begin{align}\label{eqn:euler type identity}
\mathrm{B}(b, c-b)_2 F_1(a, b ; c ; z)=\int_0^1 x^{b-1}(1-x)^{c-b-1}(1-z x)^{-a} d x
\end{align}
which holds for real numbers when $c>b>0$ and $z < 1$. We thus find
 \begin{align*}
  I   & ={(\rho^2 - d^2)^{\frac{l+1}{2}}} (\rho -d)^{k+1} \, \mathrm{B}\Big(k+1 , \mfrac{l+3}{2}\Big)\,  {}_2F_1\Big(-\mfrac{l+1}{2},\, k+1, \, \mfrac{l+1}{2} + k +2 ;\,  - \left(\mfrac{\rho - d}{\rho +d}\right) \Big)\,\\
     & = {(\rho^2 - d^2)^{\frac{l+1}{2}}} (\rho -d)^{k+1} \frac{\Gamma(k+1)\, \Gamma(\frac{l+3}{2})}{\Gamma(k+1 + \frac{l+3}{2})}\,  {}_2F_1\Big(-\mfrac{l+1}{2},\, k+1, \, \frac{l+1}{2} + k +2 ; \, -\left( \mfrac{\rho - d}{\rho +d} \right) \Big)\,.
\end{align*}

Similarly, using the change of variable $u=\lambda t$ and factorizing the integrand in term $II$ in \eqref{eqn: term I and term II}, we have
\begin{align}\nonumber
    II 
    & = \lambda^{k+1} \int_0^1 t^k \big(\rho^2 - (\lambda t + d)^2 \big)^{\frac{l+1}{2}} \,dt \\
  \nonumber  & = \lambda^{k+1} \int_0^1 t^k \big(\rho - \lambda t - d\big)^{\frac{l+1}{2}} \big(\rho + \lambda t + d\big)^{\frac{l+1}{2}} \, dt \\
 \label{eqn: factorize term II}   & = \lambda^{k+1} \big(\rho^2 - d^2\big)^{\frac{l+1}{2}} \int_0^1 t^k \Big(1 - \mfrac{\lambda }{\rho -d}\, t \Big)^{\frac{l+1}{2}} \Big(1 + \mfrac{\lambda }{\rho +d} \, t \Big)^{\frac{ l + 1}{2}} \,dt.
\end{align}
Picard's integral identity (see, e.g. \cite[\S 5.8.2]{bateman_2023_cnd32-h9x80}) states that for $c>a>0$ and $|x|, |y|<1$, 
\begin{equation}
    \label{eqn: Picard identity}
F_1\left(a, b_1, b_2, c ; x, y\right)=\frac{\Gamma(c)}{\Gamma(a) \Gamma(c-a)} \int_0^1 t^{a-1}(1-t)^{c-a-1}(1-x t)^{-b_1}(1-y t)^{-b_2} d t\,.
\end{equation}
Here, $F_1$ is the first Appell function. By applying this to \eqref{eqn: factorize term II}, using the fact that $\max\{ |\frac{\lambda}{\rho -d}| , |\frac{\lambda}{\rho + d}|\} <1$   by construction,  we find
\begin{align*}
   II & = \frac{\lambda^{k+1} (\rho^2 - d^2)^{\frac{l+1}{2}}}{k+1} F_1 \Big(k+1 , \, -\mfrac{l+1}{2},\, -\mfrac{l+1}{2} , \, k+2; \,\mfrac{\lambda}{\rho- d} , -\mfrac{\lambda}{\rho +d}\Big) \,.
\end{align*}
Putting these together, we see that $ I - II  = {(\rho^2 - d^2)^{\frac{l+1}{2}}}A/(k+1)$  for the term $A$ defined in \eqref{eqn: term A}. A completely analogous computation shows that the second integral in \eqref{eqn: Lawson comp Flint vol} is equal to  $(r^2 - h^2)^{\frac{k+1}{2}}B/(k+1)$
for the term $B$ defined in \eqref{eqn: term B}, thus establishing \eqref{eqn: lawson competitor volume}. 
\end{proof}

\begin{proof}[Proof of \eqref{eqn: lawson competitor perimeter}] The proof is analogous to that of \eqref{eqn: lawson competitor volume}.
To compute the first integral on the right-hand side of \eqref{eqn: Lawson comp Flint perim}, we have 
\begin{equation}
\begin{split} 
\rho  \int_{\lambda}^{\rho-d} u^k \big(\rho^2 -(u+d)^2\big)^{\frac{l-1}{2}} \,du  = \rho \big(\rho^2 - d^2\big)^{\frac{l-1}{2}}(III - IV).   
\end{split}
\end{equation}
where we set 
\begin{align*}
    III  &= \int_0^{\rho - d} u^k \Big(1-\mfrac{u}{\rho-d} \Big)^{\frac{l-1}{2}} \Big(1+ \mfrac{u}{\rho+d} \Big)^{\frac{l-1}{2}}\,du,\\
    IV & = \int_0^{\lambda} u^k \Big(1-\mfrac{u}{\rho-d} \Big)^{\frac{l-1}{2}} \Big(1+ \mfrac{u}{\rho+d} \Big)^{\frac{l-1}{2}}\,du.
\end{align*}
For term $III$, just as for the term $I$ in the proof of \eqref{eqn: lawson competitor volume}, we use the change of variable $u = (\rho -d) t$ and  apply the Euler identity \eqref{eqn:euler type identity} to find 
\begin{align*}
    III& = (\rho-d)^{k+1} \int_0^1 t^k (1-t)^{\frac{l-1}{2}} \Big( 1 + \mfrac{\rho-d}{\rho+d}\, t \Big)^{\frac{l-1}{2}} \\
    & = (\rho-d)^{k+1} \mathrm{B}\Big( k+1 , \mfrac{l+1}{2}\Big) \, {}_2F_1\Big( \mfrac{1-l}{2},\, k+1; \,\mfrac{l-1}{2} + k+2;\, -\left( \mfrac{\rho -d}{\rho+ d}\right) \Big)\\
    & = \frac{(\rho-d)^{k+1} \Gamma(k+1)\Gamma(\frac{l+1}{2})}{\Gamma(k+1+ \frac{l+1}{2})} \, {}_2F_1\Big( \mfrac{1-l}{2} , k+1; \mfrac{l-1}{2} + k+2;\, -\left( \mfrac{\rho -d}{\rho+ d}\right) \Big)
\end{align*}
For term $IV$, like for term $II$ in the proof of \eqref{eqn: lawson competitor volume}, we use the change of variable $u =\lambda t$ and the Picard identity \eqref{eqn: Picard identity} (since, as already noted, $\max\{ |\frac{\lambda}{\rho-d}|, |\frac{\lambda}{\rho+d}|<1$) to find
\begin{align*}
    VI & = \lambda^{k+1} \int_0^1 t^k \Big(1 - \mfrac{\lambda}{\rho-d}\, t \Big)^{\frac{l-1}{2}} \Big(1 + \mfrac{\lambda}{\rho+d} \, t\Big)^{\frac{l-1}{2}} \,dt \\ 
    & = \frac{\lambda^{k+1}}{k+1} F_1\Big(k+1, \, \mfrac{1-l}{2}, \, \mfrac{1-l}{2}, \, k+2 ; \,\mfrac{\lambda}{\rho- d}, \, -\left( \mfrac{\lambda}{\rho+ d}\right) \Big).
\end{align*}
Putting these together, we see that $V - VI = \frac{C}{k+1}$ where where $C$ is the constant in \eqref{eqn: term C}
A completely analogous argument for the second integral in \eqref{eqn: Lawson comp Flint perim} shows that 
\[
r  \int_{1}^{r-h} v^l \big(r^2 -(v+h)^2\big)^{\frac{k-1}{2}} \,dv  = \frac{r \big(r^2 - h^2\big)^{\frac{k-1}{2}} D}{l+1}
\]
where $D$ is the constant defined in \eqref{eqn: term D}.
This establishes \eqref{eqn: lawson competitor perimeter}.   
\end{proof}

\bibliography{references}
\bibliographystyle{plain}

\end{document}